\documentclass[11 pt, a4paper]{amsart}
\usepackage{amssymb,amsmath,epsfig,mathrsfs, enumerate}
\usepackage{graphicx}
\usepackage[normalem]{ulem}
\usepackage{fancyhdr}
\fancyheadoffset{0cm}
\usepackage[margin=3cm]{geometry}
\usepackage{hyperref}
\hypersetup{
 colorlinks   = true,
 urlcolor     = blue,
 linkcolor    = blue,
 citecolor   = red ,
 bookmarksopen=true
}

\newtheorem{theorem}{Theorem}[section]
\theoremstyle{plain}

\newtheorem{corollary}[theorem]{Corollary}

\newtheorem{definition}[theorem]{Definition}
\newtheorem{exmp}{Example}[section]

\newtheorem{lemma}[theorem]{Lemma}

\newtheorem{proposition}[theorem]{Proposition}
\newtheorem{remark}[theorem]{Remark}

\numberwithin{equation}{section}
\newcommand{\R}{\mathbb{R}}
\newcommand{\pa}{\partial}
\newcommand{\Om}{\Omega}

\newcommand{\Hn}{\mathbb{H}^{n}}
\newcommand{\lm}{\lambda}

\newcommand{\g}{\mathfrak{g}}
\newcommand{\Tau}{\Gamma}
\newcommand{\G}{\mathbb G}

\newcommand{\ve}{\varepsilon}
\newcommand{\cH}{\mathscr{H}}
\newcommand{\vf}{\varphi}
\newcommand{\Ri}{\mathscr R}
\newcommand{\V}{\mathscr V}
\newcommand{\s}{\sigma}
\newcommand{\U}{\mathscr U}

\title[ $\Gamma^{1,\alpha}$ boundary Schauder estimates, etc. ]{Compactness methods for $\Gamma^{1,\alpha}$ boundary Schauder estimates in Carnot groups}

\author{Agnid Banerjee}
\address{TIFR CAM, Bangalore-560065} \email[Agnid Banerjee]{agnidban@gmail.com}
\author{Nicola Garofalo}
\address{Dipartimento di Ingegneria Civile, Edile e Ambientale (DICEA) \\ Universit\`a di Padova\\ 35131 Padova, ITALY}
\email[Nicola Garofalo]{rembdrandt54@gmail.com}

\thanks{The second author was supported in part by a Progetto SID (Investimento Strategico di Dipartimento) ``Non-local operators in geometry and in free boundary problems, and their connection with the applied sciences", University of Padova, 2017.}

\author{Isidro H. Munive}
\address{CIMAT, Mexico}\email[Isidro Munive]{imunive@cimat.mx}

\begin{document}
\maketitle

\tableofcontents

\begin{abstract}
The aim of this paper is to prove  $\Tau^{1,\alpha}$ Schauder estimates near a $C^{1,\alpha}$ non-characteristic portion of the boundary for $\Tau^{0, \alpha}$ perturbations of horizontal Laplaceans in Carnot groups. This situation of minimally smooth domains presents itself naturally in the study of subelliptic free boundary problems of obstacle type, see \cite{DGP}.  
\end{abstract}

\section{Introduction}
The fundamental role of Schauder estimates (both interior and at the boundary) in the theory of elliptic and parabolic partial differential equations is well-known. In this paper we are interested in $\Gamma^{1,\alpha}$ Schauder estimates at the boundary in the Dirichlet problem for a class of second order partial differential equations in Carnot groups. The central position of such Lie groups in the analysis of the hypoelliptic operators introduced by H\"ormander in his famous paper \cite{H} was established in the 1976 work of Rothschild and Stein on the so-called \emph{lifting theorem}, see \cite{RS}. Such result represented the culmination of a visionary program laid by Stein in his address at the 1970 International Congress of Mathematicians in Nice \cite{Snice}.

To provide some perspective on the results in the present paper we mention that in  his 1981 works \cite{Je1, Je2} D. Jerison first analyzed the question of Schauder estimates at the boundary for the horizontal Laplacean in the Heisenberg group $\Hn$ (see also \cite{Je3} for a further extension to CR manifolds). Long known to physicists as the Weyl's group, $\Hn$ is an important model of a (non-Abelian) Carnot group of step $r=2$ (see Definition \ref{D:CG} below for the general notion). Jerison divided his analysis in two parts, according to whether or not the relevant portion of the boundary contains so-called \emph{characteristic points}, a notion that goes back to the pioneering works of Fichera \cite{Fi1, Fi2} (see Definition \ref{D:char} below). At such points the vector fields that form the relevant differential operator become tangent to the boundary  and thus one should expect a sudden loss of differentiability,  somewhat akin to what happens in the classical setting with oblique derivative problems. In fact, Jerison proved that there exist no Schauder boundary estimates at characteristic points! He did so by constructing a domain in $\Hn$ with real-analytic boundary that support solutions of the horizontal Laplacean $\Delta_\cH u = 0$ which vanish near a characteristic boundary point, and which near such point possess no better regularity than H\"older's. For a detailed analysis of these aspects in connection with the subelliptic Dirichlet problem we refer the reader to the papers \cite{CG, CGN}. Other relevant works are \cite{LU, CGNajm, GV}. 

On the positive side, Jerison proved in \cite{Je1} that at a non-characteristic portion of the boundary it is possible to develop a Schauder theory based on the non-isotropic Folland-Stein H\"older classes $\Gamma^{k,\alpha}$ (see Section \ref{ss:FS} below). He achieved this by adapting to $\Hn$ tools from Fourier and microlocal analysis. However, if we leave the prototype Heisenberg group $\Hn$ and we move to a general Carnot group there exists no known counterpart of the results from \cite{Je1}. 

The focus of the present paper is on Schauder estimates at a non-characteristic portion of the boundary in a Carnot group $\G$ of arbitrary step $r\ge 1$. One of the main motivations for our work comes from the analysis of some free boundary problems of obstacle type with non-holonomic constraints. In Theorem II in the paper \cite{DGP} it was proved that, in a Carnot group of step $r = 2$, under a suitable thickness condition the free boundary of a solution to the obstacle problem 
\[
\Delta_\cH u = \chi_{\{u>0\}}
\]
is locally a non-characteristic $C^{1,\alpha}$ hypersurface (for the definition of a horizontal Laplacean $\Delta_\cH$, see Definition \ref{D:sl} below). This result leads us to introduce the main contribution of this paper.   

For $k\in \mathbb N\cup\{0\}$ and $0<\alpha< 1$ we indicate with $\Gamma^{k,\alpha}$ the Folland-Stein non-isotropic H\"older classes, see Definition \ref{hf} below. We suppose that $\mathbb{A}= [a_{ij}]$ be a given $m\times m$ symmetric matrix-valued function with real coefficients and satisfying the following ellipticity condition for some $\lambda > 0$, 
\begin{equation}\label{ea0}
 \lambda \mathbb{I}_m \leq \mathbb{A}(p) \leq \lambda^{-1} \mathbb{I}_m,\ \ \ \ \ \ \ \ \ p\in \G,
 \end{equation}
where $\mathbb{I}_m$ denotes the $m \times m$ identity matrix.
If $\Om\subset \G$ is a given bounded open set, then given a point $p_0\in \pa \Om$, for any $s>0$ we set for simplicity
\begin{equation}\label{VS}
\V_s = \Om\cap B(p_0,s),\ \ \ \ \ \ \mathscr S_s = \pa\Om\cap B(p_0,s),
\end{equation}
where $B(p_0,s)$ is as in \eqref{pseudo} below.
 We consider the boundary value problem
\begin{equation}\label{dp0}
\sum_{i,j=1}^m X_i^\star (a_{ij}X_j u)  = \sum_{i=1}^m X_i^\star f_i + g\ \ \ \text{in}\ \V_s,\ \ \ \ \ u  = \phi\ \ \ \text{on}\ \mathscr S_s.
\end{equation}
The following is our main result.

\begin{theorem}\label{main}
Suppose that \eqref{ea0} hold and that for some $\alpha \in (0,1)$ the domain $\Om \subset \G$ be of class $C^{1,\alpha}$. Assume that for a given $p_0 \in \pa \Om$ and for $0<s\le 1$ the set $\mathscr S_s$ be non-characteristic. Let $u\in \mathscr{L}^{1,2}_{loc}(\V_s) \cap C(\overline{\V_s})$ be a weak solution to \eqref{dp0}, with $a_{ij}, f_i, g$ and $\phi$ satisfying the hypothesis 
\begin{equation}\label{assump1}
a_{ij} \in \Tau^{0, \alpha}(\overline{\V_s}),\ \ f_i \in \Tau^{0,\alpha}(\overline{\V_s}),\ \ g \in L^{\infty}(\overline{\V_s}),\ \ \phi \in \Tau^{1,\alpha}(\overline{\V_s}).
\end{equation} 
Then, $u\in \Tau^{1,\alpha}(\overline{\V_{s/2}})$, and we have the a priori estimate
\begin{equation}\label{ap}
[\nabla_\cH u]_{\Tau^{0,\alpha}(\V_{s/2})}  \leq \frac{C}{s^{1+\alpha}} \bigg[||u||_{L^{\infty}(\V_s)} + s^{1+\alpha} [f]_{\Tau^{0,\alpha}(\V_s)} + s^2 ||g||_{L^{\infty}(\V_s)}\bigg],
\end{equation}
where $[\nabla_\cH u]_{\Tau^{0,\alpha}(\V_{s/2})}$ is the H\"older seminorm of the horizontal gradient of $u$, see \eqref{hg} and \eqref{semi} below, and $C = C(\alpha, \G,\lambda,[a_{ij}]_{\Tau^{0,\alpha}(\V_1)}, \Om)>0$.
\end{theorem}

Our approach to Theorem \ref{main} is based on an adaptation of compactness arguments that have their roots in the seminal paper of Caffarelli \cite{Ca}. For such arguments the homogeneous  structure of a Carnot group based on the presence of a family of non-isotropic dilations plays a critical role. We mention that related ideas have been previously employed in the subelliptic setting for interior Schauder estimates, see e.g. \cite{CH, GL, XZ}. With respect to such works, our situation is complicated by the presence of the Dirichlet condition, whose handling  has required some delicate arguments. Also, differently from the classical approach to boundary Schauder estimates, we do not flatten the boundary. The successful implementation of a coordinate-free approach to the non-Euclidean setting of this paper, in which there is lack of ellipticity at every point, provides an alternative more geometric viewpoint on Schauder estimates at the boundary which, we believe, is robust enough to be applicable to higher-order Schauder estimates. This is an aspect to which we plan to come back in a future study. 

In connection with our work we mention the recent paper \cite{BCC}, in which the authors have established an interesting $\Tau^{2, \alpha}$ boundary Schauder estimate at the non-characteristic boundary for $C^\infty$ domains in Carnot groups. We emphasize that our Theorem \ref{main} is substantially different in nature from their main result, and cannot be deduced from it. The reason for this is twofold: 1) in \cite{BCC} the authors work with smooth domains. We work with $C^{1,\alpha}$ domains. Since, as we have said, our main motivation is the obstacle problem, this distinction is essential. In a free boundary problem one does not a priori know the higher regularity of the free boundary, and consequently the results in \cite{BCC} cannot be applied in such situations;  2) in \cite{BCC} the authors work with non-divergence form equations and $\Gamma^\alpha$ coefficients. Instead, we work with divergence form equations with $\Gamma^\alpha$ coefficients. If one writes our equations in non-divergence form, one obtains first-order terms with singular coefficients which cannot be handled by the methods in [BCC].

To provide the reader with an additional perspective on the present work we mention that, in the classical obstacle problem, the a priori knowledge of the   $C^{1,\alpha}$ boundary smoothness of solutions vanishing on a portion  of a  $C^{1, \alpha}$ boundary is crucially needed for improving the regularity of the  free boundary from $C^{1, \alpha}$ to $C^\infty$. The only known methods to implement such bootstrapping procedure are that of Kinderlehrer and Nirenberg in \cite{KN77} (see also \cite{KNS}), based on the hodograph transformation combined with the thickness condition in \cite{Ca1, Ca2}, or the more recent  higher-order boundary Harnack  result of De Silva and Savin, see \cite{DS1, DS2}. Given the above mentioned $C^{1, \alpha}$ smoothness of the free boundary established in \cite{DGP} for the subelliptic obstacle problem, it is reasonable to think that the regularity result in the present paper should provide a reasonable ground for investigating  whether the ideas in \cite{KN, DS1, DS2}  can be extended to the setting of Carnot groups.

This paper is organized as follows. In Section 2, we introduce the relevant notations and collect some basic properties of Carnot groups. In Section 3, we   gather some  known preliminary results and establish some regularity estimates  which are relevant to our analysis.  In Section 4, we finally prove our main result, Theorem \ref{main}.

\section{Preliminaries}\label{S:prel}

In this section we collect some basic properties of Carnot groups which will be used in the rest of the paper. We begin with the relevant definition.

\begin{definition}\label{D:CG}
Given $r\in \mathbb N$, a Carnot group of step $r$ is a simply-connected real Lie group $(\G, \circ)$ whose Lie algebra $\g$ is stratified and $r$-nilpotent. This means that there exist vector spaces $\g_1,...,\g_r$ such that  
\begin{itemize}
\item[(i)] $\g=\g_1\oplus \dots\oplus\g_r$;
\item[(ii)] $[\g_1,\g_j] = \g_{j+1}$, $j=1,...,r-1,\ \ \ [\g_1,\g_r] = \{0\}$.
\end{itemize}
\end{definition}
We note that when $r = 1$ then the group is Abelian, and we are back into the familiar Euclidean situation. By the assumption that $\G$ be simply-connected we know that the exponential mapping $\exp: \g \to \G$ is a global analytic diffeomorphism onto, see \cite{V, CGr}. We will use this global chart to identify the point $p = \exp \xi\in \G$ with its logarithmic preimage $\xi\in \g$. 

Once the bracket relations at the level of the Lie algebra are assigned, the group law is too. This follows from the Baker-Campbell-Hausdorff formula, 
see, e.g., sec. 2.15 in \cite{V},
\begin{equation}\label{BCH}
\exp(\xi) \circ \exp(\eta) = \exp{\bigg(\xi + \eta + \frac{1}{2}
[\xi,\eta] + \frac{1}{12} \big\{[\xi,[\xi,\eta]] -
[\eta,[\xi,\eta]]\big\} + ...\bigg)},
\end{equation}
where the dots indicate commutators of order three and higher.  Furthermore, since by (ii) in Definition \ref{D:CG} above all commutators of order $r$ and higher are trivial, in every Carnot group the Baker-Campbell-Hausdorff series in the right-hand side of \eqref{BCH} is finite.

According to (ii) the first layer $\g_1$ plays a special role since it bracket-generates the whole Lie algebra $\g$. It is traditionally referred to as the \emph{horizontal layer} of $\g$.  We assume that a scalar product $\langle \cdot,\cdot \rangle$ is given on $\g$ for which the $\g_j's$ are mutually orthogonal. We let $m_j = \dim \g_j$, $j = 1, \dots, r,$ and denote by
\begin{equation}\label{N}
N = m_1 + \ldots + m_r
\end{equation}
the topological dimension of $\G$. For notational simplicity we will hereafter denote with $m=m_1$ the dimension of the horizontal layer $\g_1$.

Every Carnot group is naturally equipped with \emph{translations} and \emph{dilations}. Using the group law $\circ$ we can respectively define the left- and right-translations in $\G$ by an element $p_0 \in \G$ by
\begin{equation}\label{trans}
L_{p_0}(p) = p_0 \circ p,\quad\ \ \ \ \ \  R_{p_0}(p) = p\circ p_0.
\end{equation}
Given a function $f:\G\to \R$, the action of $L_{p_0}$ and $R_{p_0}$ on $f$ is defined by
\[
L_{p_0} f(p) = f(L_{p_0}(p)),\ \ \ \ \ \ \ R_{p_0} f(p) = f(R_{p_0}(p)),\ \ \ \ \ p\in \G.
\]
A vector field $Y$ on $\G$ is called left-invariant (or right-invariant) if for any $f\in C^\infty(\G)$ and any $p_0\in \G$ one has
\[
Y(L_{p_0} f) = L_{p_0}(Yf),\ \ \ \ \ \ \ \ \ Y(R_{p_0} f) = R_{p_0}(Y f).
\]

To define the dilations in a Carnot group $\G$ one assigns the formal degree $j$ to the $j$-th layer $\g_j$ of the Lie algebra. Then, a family of non-isotropic dilations $\Delta_{\lambda} : \g \rightarrow \g$ is defined by setting for every $\xi = \xi_1 + ... + \xi_r\in \g$, with $\xi_j \in \g_j$, $j = 1,\ldots,r$, 
\begin{equation}\label{dil}
\Delta_{\lambda}\xi = \lambda\xi_1 +\cdots+\lambda^r\xi_r. 
\end{equation}
One then uses the exponential mapping to lift \eqref{dil} to a one-parameter family $\{\delta_\lm\}_{\lm>0}$ in the group $\G$ by letting for $\lm>0$
\begin{equation}\label{dilg}
\delta_{\lm}(p) = \exp\circ\Delta_{\lm}\circ\exp^{-1}(p), \ \ \ \ \ \ p\in\G.
\end{equation}
The dilations are group automorphisms, and thus we have for any $p, p'\in \G$ and $\lambda>0$
\begin{equation}\label{ga}
(\delta_\lambda(p))^{-1} = \delta_\lambda(p^{-1}),\ \ \ \ \ \ \ \ \ \ \delta_\lambda(p) \circ \delta_\lambda(p') = \delta_\lambda(p\circ p').
\end{equation}
The \emph{homogeneous dimension} of $\G$ with respect to the dilations \eqref{dilg} is defined as follows
\begin{equation}\label{QG}
Q = \sum_{j = 1}^r j \operatorname{dim} \g_j.
\end{equation}
The name comes from the fact that the bi-invariant Haar measure $dp$ on $\G$ (which is obtained by pushing forward via the exponential mapping the Lebesgue measure on $\g$) interacts with $\{\delta_\lm\}_{\lm>0}$ according to the formula
\begin{equation}\label{haar}
d \circ \delta_\lm(p) = \lm^Q dp.
\end{equation}
We note that in the non-Abelian setting when $r>1$ the number $Q$ in \eqref{QG} is strictly bigger than the topological dimension $N$ of $\G$. Such number plays a pervasive role in the analysis and geometry of $\G$, see \cite{F}.

Given a function $f:\G\to \R$, the action of $\{\delta_\lm\}_{\lm>0}$ on $f$ is defined by
\[
\delta_\lm f(p) = f(\delta_\lm(p)),\ \ \ \ \ \ \ \ \ \ \ p\in \G.
\]
In view of \eqref{haar}, the action of $\{\delta_\lm\}_{\lm>0}$ on a distribution $T\in \mathscr D'(\G)$ is defined by the equation
\begin{equation}\label{Tlambda}
<\delta_\lambda T, \varphi> = <T,\lambda^{-Q} \delta_{\lambda^{-1}} \varphi>,\quad\quad\quad\quad \varphi\in \mathscr D(\G).
\end{equation}

\begin{definition}\label{D:hom}
Let $\kappa\in \R$. A function $f:\G \to \R$ is called \emph{homogeneous} of degree $\kappa$ if for every $\lm >0$ one has
\[
\delta_\lm f = \lm^\kappa f.
\]
A vector field $Y$ on $\G$ is called homogeneous of degree $\kappa$ if for every $f\in C^\infty(\G)$ one has
\[
Y(\delta_\lm f) = \lm^\kappa \delta_\lm(Yf).
\]
A distribution $T\in \mathscr D'(\G)$ is called homogeneous of degree $\kappa$ if we have in $\mathscr D'(\G)$
\[
\delta_\lm T = \kappa T.
\]
\end{definition} 

\subsection{Horizontal Laplaceans}\label{SS:hl}

With any given orthonormal basis $\{e_1,\ldots,e_m\}$ of $\g_1$ we associate a family $\{X_1,\ldots,X_m\}$ of left-invariant vector fields on $\G$ by letting for $j=1,...,m$ and $p\in \G$
\begin{equation}\label{xjs}
X_j (p) = dL_p(e_j),
\end{equation}
where we have denoted by $dL_p$ the differential of $L_p$. The vector fields $\{X_1,\ldots,X_m\}$ constitute a basis for the so-called horizontal sub-bundle $\cH$ of the tangent bundle $T\G$. Given a point $p\in \G$, the fiber of $\cH$ at $p$ is given by 
\begin{equation}\label{fiber}
\cH_p = dL_p(\g_1).
\end{equation}
We note that the action of $X_j$ on a function $f\in C^\infty(\G)$ is specified by the Lie derivative
\begin{equation}\label{fields}
X_jf(p) = \lim_{t\rightarrow 0} \frac{f(p \exp (te_j)) - f(p)}{t} = \frac{d}{dt}\Big|_{t=0} f(p \exp (te_j)). 
\end{equation}
In this paper given vector fields defined as in \eqref{xjs} above, we will routinely assume that $\G$ has been endowed with a left-invariant Riemannian metric with respect to which the system $\{X_1,...,X_m\}$ is orthonormal. Given a function $f\in C^\infty(\G)$, we denote by 
\begin{equation}\label{hg}
\nabla_\cH f = \sum_{j=1}^m X_j f X_j
\end{equation}
its \emph{horizontal gradient}, i.e., the projection of the Riemannian connection $\nabla$ of $\G$ onto the horizontal bundle $\cH$. We clearly have
\begin{equation}\label{lhg}
|\nabla_{\cH} f|^2 = \sum_{j=1}^m (X_j f)^2.
\end{equation}

From \eqref{fields} and the properties of the dilations $\{\delta_\lm\}_{\lm>0}$, one can easily verify the following.

\begin{lemma}\label{L:xjhom}
For every $j = 1,...,m$ the left-invariant vector fields $X_j$ defined by \eqref{xjs} are homogeneous of degree $\kappa = 1$.
\end{lemma} 

Before proceeding we note that, if $X_j^\star$ indicates the formal adjoint of $X_j$ in $L^2(\G)$, then $X_j^{\star} = -X_j$, see \cite{F}.

\begin{definition}\label{D:sl}
The \emph{horizontal Laplacean} associated with an orthonormal basis $\{e_1,...,e_m\}$ of the horizontal layer $\g_1$ is the left-invariant second-order partial differential operator in $\G$ defined by
\begin{equation}\label{subl}
 \Delta_\cH = -\sum^m_{j=1}X^{\star}_{j}X_j=\sum^m_{j=1}X^2_j,
\end{equation}
where $\{X_1,...,X_m\}$ are as in \eqref{fields} above.
\end{definition}

By Lemma \ref{L:xjhom} and \eqref{subl} it is clear that every horizontal Laplacean is homogeneous of degree $\kappa = 2$, i.e., one has for every $f\in C^\infty(\G)$
\begin{equation}\label{delinv}
\Delta_\cH(\delta_\lm f) = \lm^2 \delta_\lm(\Delta_\cH f).
\end{equation}
Furthermore, by the assumptions (i) and (ii) in Definition \ref{D:CG} one immediately sees that the system $\{X_1,...,X_m\}$ satisfies the finite rank condition 
\[
\text {rank Lie} [X_1,\ldots,X_m] \equiv N,
\] 
therefore by  H\"ormander's theorem \cite{H} the operator $\Delta_\cH$ is hypoelliptic. However, when the step $r$ of $\G$ is $>1$ this operator fails to be elliptic at every point $p\in \G$.

\subsection{Intrinsic distance and gauge pseudo-distance}\label{SS:intrinsic}

In a Carnot group there exists a privileged left-invariant distance called the \emph{control} or \emph{Carnot-Carath\'eodory distance} $d_C(p,p')$ associated with the horizontal subbundle $\cH$. The properties of such distance have been studied in the celebrated paper \cite{NSW}, see also \cite{Be} and Chapter 4 in \cite{Gems}. A piecewise $C^1$ curve $\gamma:[0,T]\to \G$ is called \emph{horizontal} if there exist piecewise continuous functions $a_i:[0,T]\to \G$ with $\sum_{i=1}^m |a_i| \le 1$ such that
\[
\gamma'(t)=\sum_{i=1}^m a_i(t)X_i(\gamma(t)).
 \]
We define the \emph{horizontal length} of $\gamma$ as $\ell_\cH(\gamma) = T$. Given two points $p, p' \in \G$ denote by  $\mathscr S(p,p')$ the collection of all horizontal curves $\gamma:[0,T]\to \G$ such that $\gamma(0)=p$ and $\gamma(T)=p'$. By the theorem of Chow-Rashevsky we know that $\mathscr S(p,p')\not= \varnothing$ for every $p, p'\in \G$. We define
\[
d_C(p,p')=\underset{\gamma\in \mathscr S(p,p')}{\inf} \ell_\cH(\gamma).
\]
The Carnot-Carath\'eodory 
metric $d_C(p,p')$ is equivalent to a more explicitly
defined pseudo-distance function, called the  \emph{gauge pseudo-distance}, defined as follows.
Let $||\cdot||$ denote the Euclidean distance to the origin in $\g$. For $\xi = \xi_1 +\cdots+\xi_r \in \g$, $\xi_j \in \g_j$, $j = 1,\ldots,r$, one lets
\begin{equation}\label{homd}
|\xi|_{\g} = \left(\sum^r_{j=1} ||\xi_j||^{\frac{2r!}{j}}\right)^{2r!}, \ \ \ \ \quad |p|_{\G}=|\exp^{-1} p|_{\g}\ \quad p \in \G.
\end{equation}
The function $p\to |p|_\G$ is called the non-isotropic group gauge. We will routinely drop the subscript $\G$ and simply denote it by $|p|$. Notice that from \eqref{homd} and \eqref{dilg} we have for any $\lambda>0$
\begin{equation}\label{gau}
|\delta_\lambda(p)| = \lambda |p|.
\end{equation}
The \emph{gauge pseudo-distance} in $\G$ is defined by
\begin{equation}\label{rho0}
d(p,p') = |p^{-1} \circ p'|.
\end{equation}
The function $d(p,p')$ has all the properties of a distance, except the triangle inequality, which is satisfied with a universal constant different from $1$ in the right-hand side, see \cite{FS}. 
Notice that from \eqref{rho0}, \eqref{ga} and \eqref{gau} we have for any $\lambda>0$ 
\begin{equation}\label{gauge2}
d(\delta_\lambda(p),\delta_\lambda(p')) = \lambda d(p,p').
\end{equation}
It is well-known, see for instance Proposition 4.28 in \cite{Gems}, that there exist universal constants $C_1, C_2>0$ such that for $p, p'\in \G$ one has
\begin{equation}\label{rhod}
C_1 d(p,p') \le d_C(p,p') \le C_2 d(p,p').
\end{equation}
An immediate consequence of \eqref{rhod} is the following pseudo-triangle inequality for $d$ 
\begin{equation}\label{t}
d(p,p') \leq C_0 (d(p,p'') + d(p'',p'))
\end{equation}
for all $p, p', p''\in \G$, and a universal $C_0>0$. 
We respectively denote with
\begin{equation}\label{pseudo}
B_C(p,R) = \{p' \in\G\ |\ d_C(p',p)<R\},\ \ \ \ \ \ B(p,R)=\{p' \in\G\ |\ d(p',p)<R\}, 
\end{equation}
the metric and the gauge pseudo-ball centered at $p$ with radius $R$. When the center is the group identity $e$, we will simply write $B_C(R)$ and $B(R)$ instead of $B_C(e,R)$ and $B(e,R)$. 
 
Hereafter, we indicate with $|E| = \int_E dp$ the Haar measure of a set $E\subset \G$.
One easily recognizes from \eqref{haar}, \eqref{gauge2} and \eqref{gau} that, with $\omega_C = \omega_C(\G) = |B_C(1)|>0$ and $\omega = \omega(\G) = |B(1)|>0$, one has for every $p \in \G$ and $R > 0$, 
\begin{equation}\label{volb}
|B_C(p,R)| =\omega_C R^Q,\ \ \ \ \ \ \ \ \ \ \ |B(p,R)| =\omega R^Q.
\end{equation}

We close this subsection by recalling the following result from \cite{NSW}. Denote by $d_\Ri(x,y)$ the left-invariant Riemannian distance in $\G$. 

\begin{lemma}\label{L:dis}
For every connected  $\Omega \subset\subset \G$ there exist $C,\varepsilon
>0$ such that
\begin{equation}\label{xy}
C d_\Ri(p,p') \leq d_C(p,p')\leq C^{-1} d_\Ri(p,p')^\varepsilon,
\end{equation}
for $p, p' \in \Omega$.
\end{lemma}


\subsection{Homogeneous polynomials}\label{SS:wp}

Henceforth, the notation $\{e_{j,1}, \dots, e_{j,m_j}\}$, $j = 1,\dots, r,$ will indicate a fixed orthonormal basis of the $j$-th layer $\g_j$ of the Lie algebra $\g$, keeping in mind that we reserve the notation $\{e_1,...,e_m\}$ for $\{e_{1,},...,e_{1,m_1}\}$. Since the exponential map $\exp : \g \rightarrow \G$ is a global analytic diffeomorphism, it allows to define analytic maps $\xi_i : \G \rightarrow \g_j$, $j = 1,\ldots,r$, by letting $p = \exp\left(\xi_1(p)+\ldots+ \xi_r(p)\right)$. For $p \in \G$, the projection of the \emph{logarithmic coordinates} of $p$ onto the layer $\g_j$, $j = 1,\ldots, r,$ are defined as follows
\begin{equation}
\label{coor}
x_{j,s}(p) = \langle \xi_j(p),e_{j,s}\rangle, \quad s = 1,\ldots,m_j,
\end{equation}
where again for ease of notation we let $(x_1(p),...,x_m(p)) = (x_{1,1}(p),...,x_{1,m}(p))$ indicate the horizontal coordinates of $p$.
Whenever convenient we will routinely omit the dependence in $p$, and simply identify $p$ with its logarithmic coordinates
\begin{equation}\label{pvar}
p \cong (x_1,...,x_m, x_{2,1},...,x_{2,m_2},....,x_{r,1},...,x_{r,m_r}).
\end{equation}
To simplify this lengthy notation we will group variables according to their layer in the following way
\begin{equation}\label{xis}
\xi_1 = (x_1,...,x_m),\ \xi_2 = (x_{2,1},...,x_{2,m_2}), . . . , \xi_r = (x_{r,1},...,x_{r,m_r}).
\end{equation}
Furthermore, at times it will be expedient to break the variables by grouping those in the horizontal layer into a single one as follows
$x = x(p) \cong \xi_1 = (x_1,...,x_m)$, and indicate with $y = y(p)$ the $(N - m)-$dimensional vector
\[
y \cong (\xi_2,...,\xi_r) = (x_{2,1},...,x_{2,m_2},....,x_{r,1},...,x_{r,m_r}).
\]
In this case, we will write $z = (x,y)$.

The reader should see Section C on p. 20 in \cite{FS}. Suppose that for every $j=1,...,r$ we assign a multi-index
\[
\beta_j = (\beta_{j,1},...,\beta_{j,m_j}) \in (\mathbb N\cup \{0\})^{m_j}.
\]
Keeping \eqref{xis} in mind we consider the monomial 
\[
\xi_j^{\beta_j} = x_{j,1}^{\beta_{j,1}}\ . . .\ x_{j,m_j}^{\beta_{j,m_j}}.
\]
We can then form a multi-index $I\in (\mathbb N \cup \{0\})^N$ as follows
\begin{equation}\label{I}
I = (\beta_1,...,\beta_r) = (\beta_{j,1},...,\beta_{j,m},...,\beta_{r,1},...,\beta_{r,m_r}),
\end{equation}  
where $N$ is as in \eqref{N}.
We define  
\begin{equation}\label{hl}
|I| = \sum_{j=1}^r |\beta_j| = \sum_{j=1}^r  \sum_{k=1}^{m_j} \beta_{j,k},\ \ \ \ \ d(I) = \sum_{j=1}^r j |\beta_j| = \sum_{j=1}^r j \sum_{k=1}^{m_j} \beta_{j,k}, 
\end{equation}
and call $d(I)$ the \emph{homogeneous length} of $I$. 
Given $I$ as in \eqref{I}, consider the monomial $z^I$ defined by 
\[
z^I = \xi_1^{\beta_1}\ ...\ \xi_r^{\beta_r} = \prod_{j=1}^r x_{j,1}^{\beta_{j,1}}\ . . .\ x_{j,m_j}^{\beta_{j,m_j}}.
\]
Identifying $p\in \G$ with its logarithmic coordinates $z$, it is clear that the function $f(p) = z^I$ is homogeneous of degree $\kappa = d(I)$ according to Definition \ref{D:hom}.

\begin{definition}
\label{pol}
A \emph{homogeneous polynomial} in $\G$  is a function $P:\G\rightarrow \R$ which in the logarithmic coordinates $z = (x,y)$ can be expressed as
\[
P(z)=\sum_I a_I z^I,
\]
where $a_I\in \R$. The \emph{homogeneous degree} of $P$ is the largest $d(I)$ for which the corresponding $a_I\not= 0$. For any  $\kappa\in \mathbb{N}\cup \{0\}$ we denote by $\mathscr{P}_\kappa$ the set of homogeneous polynomials in $\G$ of homogeneous degree less or equal to $\kappa$. 
\end{definition}

The space $\mathscr{P}_\kappa$ is invariant under left- and right-translation. We note explicitly that, in the logarithmic coordinates, the elements of $\mathscr P_1$ are represented as 
\[
P = \alpha_0 + \sum_{j=1}^m \alpha_j x_j.
\]

\subsection{The Folland-Stein H\"older classes}\label{ss:FS}

A basic tool in this paper are the intrinsic H\"older classes $\Tau^{k, \alpha}$ introduced by Folland and Stein, see Section 5 in \cite{F} and especially \cite{FS}. The functions in these classes are H\"older continuous, along with a certain number of derivatives along the horizontal directions, with respect to the Carath\'eodory distance $d_C(p,p')$. In view of \eqref{rhod} we can, and will, equally well employ the gauge pseudo-distance $d(p,p')$ in $\G$. For the proofs of the results in this subsection we refer the reader to Section C on p. 20 in \cite{FS} and to Chapter 20 in \cite{BLU}. 
 
\begin{definition}\label{hf}
Let $0< \alpha \leq 1$. Given an open set $\Om \subset \G$ we say that $u:\Om\to \R$ belongs to $\Tau^{0, \alpha}(\Om)$ if there exists a positive constant $M$ such that for every $p, p' \in \Om$, 
\[
|u(p) - u(p')| \leq M\ d(p,p')^{\alpha}.
\]
We define the seminorm 
\begin{equation}\label{semi}
[u]_{\Tau^{0,\alpha}(\Om)}= \underset{\underset{p \neq p'}{p, p'\in \Om}}{\sup} \frac{|u(p)-u(p')|}{d(p,p')^{\alpha}}.
\end{equation}
Given $k\in \mathbb N$, the spaces $\Tau^{k, \alpha}(\Om)$ are defined inductively as follows: we say that $u \in \Tau^{k, \alpha}(\Om)$ if $X_i u \in \Tau^{k-1,\alpha}(\Om)$ for every $i=1, .., m$. 
\end{definition}
We note that from \eqref{dilg} and \eqref{gauge2} one has for any $\lambda>0$ 
\begin{equation}\label{udil}
[\delta_\lambda u]_{\Tau^{0,\alpha}(\delta_{\lambda^{-1}}(\Om))} = \lambda^\alpha [u]_{\Tau^{0,\alpha}(\Om)}.
\end{equation}
We now introduce the relevant notion of Taylor polynomials. 

\begin{definition}\label{D:tp}
Suppose $p\in \G, \kappa\in \mathbb{N}\cup\{0\}$, and $f$ is a function whose derivatives $X^If$ are continuous functions in a neighborhood of $p$ for $|I|\leq \kappa$. The \emph{left Taylor polynomial} of $f$ at $p$ of weighted degree $\kappa$ is the unique $P\in \mathscr{P}_\kappa$ such that $X^IP(e)=X^If(p)$ for $|I|\leq \kappa$.
\end{definition}

\subsection{The characteristic set}\label{SS:char}

We recall that an open set $\Om \subset \G$ is said to be of class $C^1$ if for every $p_0\in \pa\Om$ there exist a neighborhood $U_{p_0}$ of $p_0$, and a function $\vf_{p_0} \in C^1(U_{p_0})$, with $|\nabla \vf_{p_0}| \geq \alpha > 0$ in $U_{p_0}$, such that
\begin{equation}\label{ome}
\Om \cap U_{p_0} = \{p\in U_{p_0} \mid \vf_{p_0}(p)<0\}, \quad\pa\Om\cap U_{p_0} =\{p\in U_{p_0} \mid \vf_{p_0}(p)=0\}.
\end{equation}
At every point $p\in \partial \Om \cap U_{p_0}$ the outer unit normal is given by
\[
\nu(p) = \frac{\nabla \vf_{p_0}(p)}{|\nabla \vf_{p_0}(p)|},
\] 
where $\nabla$ denotes the Riemannian gradient.

\begin{definition}\label{D:char}
Let  $\Om \subset \G$ be an open set of class $C^1$. A point $p_0\in \pa\Om$ is called \emph{characteristic} if with $\cH_{p_0}$ as in \eqref{fiber}, one has
\begin{equation}\label{cp}
\nu(p_0) \perp \cH_{p_0}.
\end{equation}
The \emph{characteristic set} $\Sigma = \Sigma_{\Om}$ is the collection of all characteristic points of $\Om$. A boundary point $p_0 \in \pa \Om \setminus \Sigma$ will be referred to as a \emph{non-characteristic} boundary point. 
\end{definition}

We note explicitly that \eqref{cp} is equivalent to saying that,
given any couple $U_{p_0}, \vf_{p_0}$ as in \eqref{ome}, and a family of vector fields $\{X_1,...,X_m\}$ as in \eqref{xjs} above, one has 
\begin{equation}\label{cp2}
X_1\vf_{p_0}(p_0) = 0, \ldots, X_m\vf_{p_0}(p_0) = 0,
\end{equation}
or equivalently
\begin{equation}\label{cp3}
 |\nabla_\cH \vf_{p_0}(p_0)| = 0.
\end{equation}

Bounded domains generically have non-empty characteristic set if they have a trivial topology. For instance, in the Heisenberg group $\Hn$ every bounded $C^1$ open set which is homeomorphic to the sphere $\mathbb S^{2n}\subset \R^{2n+1}$ must have at least one characteristic point. A torus obtained by revolving a sphere around the $t$-axis in $\Hn$ provides an example of a non-characteristic domain. For these facts we refer the reader to \cite{CG}. An example of an unbounded non-characteristic domain is the following. 

\begin{exmp}
 For a fixed vector $a \in \g_1 \setminus \{0\}$ and for $\lambda \in \R$, consider the so-called \emph{vertical half-space}
 \[
H^+_a = \{p\in\G\mid <x(p),a>> \lambda\}.
 \]
Then, using the global defining function $\vf(p) = \lm - <x(p),a>$, one has $|\nabla_\cH \vf(p)| \equiv |a|>0$, and thus in view of \eqref{cp3} one has $\Sigma = \Sigma_{H^a_+} = \varnothing$. 
\end{exmp}

\section{Some basic regularity estimates}\label{S:basic}

In this section we collect some basic regularity estimates, both interior and at the boundary, that we need in our work. We begin with an interior Morrey-Campanato type result of Xu and Zuily in \cite{XZ}, see Theorem \ref{T:XZ} and Corollary \ref{intr} below. After that we state an H\"older estimate at a $C^{1,\alpha}$ non-characteristic portion of the boundary which follows from a result of Danielli in \cite{D}, see Proposition \ref{hol} below. We close the section by proving a smoothness result, this time at a $C^\infty$ non-characteristic portion of the boundary, see Theorem \ref{KNc} below. The main step of the proof consists in showing that a weak solution has bounded horizontal gradient up to the boundary. Once that is accomplished, we appeal to results of Kohn-Nirenberg \cite{KN} or Derridj \cite{De} to infer the smoothness  of the weak solution up to the boundary.

In what follows we assume that  an orthonormal basis $\{e_1,...,e_m\}$ of $\g_1$ has been fixed, and that $\{X_1,...,X_m\}$ are correspondingly defined by \eqref{xjs}. For an  open set $\Om\subset \G$  we denote by $\mathscr{L}^{1,p}(\Om)$, where $1\leq p\leq\infty$,  the Banach space $\{f \in L^p(\Om)\ | \ X_jf \in L^p(\Om), j = 1,\ldots,m\}$ endowed with its natural norm
\[
\|f\|_{\mathscr{L}^{1,p}(\Om)} = \|f\|_{L^{p}(\Om)} +   \sum_{j=1}^m \|X_jf\|_{L^{p}(\Om)}.
\]
The local space $\mathscr{L}^{1,p}_{loc}(\Om)$ has the usual meaning. We also denote by $\mathscr{L}^{1,p}_0(\Om) = \overline{C^\infty_0(\Om)}^{||\cdot||_{\mathscr{L}^{1,p}(\Om)}}$.
We note that the definition of the spaces $\mathscr{L}^{1,p}(\Om)$, $\mathscr{L}^{1,p}_{loc}(\Om)$, $\mathscr{L}^{1,p}_0(\Om)$ does not depend on the particular orthonormal basis $\{e_1,...,e_m\}$ that we pick.

Throughout this section for a given bounded open set $\Om\subset \G$ we consider a weak solution $u$ to the equation
\begin{equation}\label{re}
\sum_{i,j=1}^m X_i^\star (a_{ij} X_j u)= \sum_{i=1}^m X_i^\star f_i + g,
\end{equation}
where on the functions $a_{ij}, f_1,...,f_m, g$ we will make suitable assumptions. In particular, we always assume that $\mathbb{A}= [a_{ij}]$ be a given $m\times m$ symmetric matrix with real coefficients and satisfying the ellipticity condition \eqref{ea0} for some $\lambda > 0$. 
We also always assume that $f_i, g\in L^\infty(\Om)$. Under these assumptions we recall that by a weak solution to \eqref{re} we mean a function $u\in \mathscr{L}^{1,2}_{loc}(\Om)$ such that for any $\zeta\in \mathscr{L}^{1,2}_0(\Om)$, with supp$\ \zeta\subset \Om$, one has
\[
\sum_{i,j=1}^m \int_\Om a_{ij} X_i u X_j \zeta =  \int_\Om \left(\sum_{i=1}^m f_i X_i \zeta + g \zeta\right).
\]
When we combine \eqref{re} with a Dirichlet boundary condition $u = \phi$ on $\pa \Om$, then by a solution of the boundary value problem we mean that $u$ is a weak solution of \eqref{re}, that $\phi\in \mathscr{L}^{1,2}(\Om)$, and that $u - \phi\in \mathscr{L}^{1,2}_0(\Om)$.

We will need the following interior regularity result of Xu and Zuily in \cite{XZ}.

\begin{theorem}\label{T:XZ}
Let $u\in \mathscr{L}^{1,2}_{loc}(B(p_0,1))$ be a weak solution to \eqref{re} in a gauge ball $B(p_0,1)$. Suppose furthermore that for some $0<\alpha<1$ one have $a_{ij} \in \Tau^{0, \alpha}(B(p_0,1))$, $f= (f_1,..., f_m) \in \Tau^{0, \alpha}(B(p_0,1))$, and $g \in L^{\infty}(B(p_0,1))$. Then, $u \in \Tau^{1,\alpha}(B(p_0,1/2))$ and there exists a constant $C = C(\G,\lambda,[a_{ij}]_{\Tau^{0, \alpha}(B(p_0,1))})>0$ such that the following estimates hold
\begin{equation}\label{i10}
 ||\nabla_\cH u||_{L^{\infty}(B(p_0,1/2))}  \leq  C \left\{||u||_{L^{\infty}(B(p_0,1))} + [f]_{\Tau^{0, \alpha}(B(p_0,1))} + ||g||_{L^{\infty}(B(p_0,1))}\right\},
\end{equation}
and
\begin{equation}\label{i100}
[\nabla_\cH u]_{\Tau^{0,\alpha}(B(p_0,1/2))}  \leq C \left\{||u||_{L^{\infty}(B(p_0,1))} + [f]_{\Tau^{0,\alpha}(B(p_0,1))} + ||g||_{L^{\infty}(B(p_0,1))}\right\}, 
\end{equation}
where  $[\cdot]_{\Tau^{0, \alpha}(B(p_0,1))}$ represents the  seminorm  defined by \eqref{semi} above.
\end{theorem}

As a corollary, we have the following corresponding estimate at any given scale $s>0$.

\begin{corollary}\label{intr}
Suppose that $a_{ij} \in \Tau^{0, \alpha}(B(p_0,s))$. Given $0<s\le 1$, let $u$ be a weak solution to \eqref{re} in $B(p_0,s)$, where $f= (f_1,...,f_m) \in \Tau^{0, \alpha}(B(p_0,s))$ and $g \in L^{\infty}(B(p_0,s))$. Then, $u \in \Tau^{1,\alpha}(B(p_0,s/2))$ and satisfies the following estimates
 \begin{equation}\label{i13}
 ||\nabla_\cH u||_{L^{\infty}(B(p_0,s/2))}  \leq  \frac{C}{s} (||u||_{L^{\infty}(B(p_0,s))} + s^{1+\alpha} [f]_{\Tau^{0,\alpha}(B(p_0,s))} + s^2 ||g||_{L^{\infty}(B(p_0,s))}),
\end{equation}
and
\begin{equation}\label{i133}
[\nabla_\cH u]_{\Tau^{0,\alpha}(B(p_0,s/2)}  \leq \frac{C}{s^{1+\alpha}} (||u||_{L^{\infty}(B(p_0,s))} + s^{1+\alpha} [f]_{\Tau^{0,\alpha}(B(p_0,s))} + s^2 ||g||_{L^{\infty}(B(p_0,s))}), 
\end{equation}
 where $C>0$ is as in the previous theorem.
\end{corollary}

\begin{proof}
By left translation we may assume that  $p_0=e$. Then, with $\delta_s$ is as in \eqref{dilg} we note that the function $ u_s(p)= u(\delta_{s} p)$ solves in $B(e,1)$ the rescaled equation
\[
\sum_{i,j=1}^m X_i^\star( a_{ij, s} X_j u_s)= \sum_{i = 1}^m X_i^\star f_{i,s} + g_s,
\]
where
\[
a_{ij,s}(p)= a_{ij}(\delta_s p),\ \ \ \ f_{i,s}(p)= s f_i(\delta_s p),\ \ \  \ g_s= s^2 g(\delta_s p).
\]
By the estimates in \eqref{i10}, \eqref{i100} we obtain
\begin{equation}\label{ok}
||\nabla_\cH u_s||_{L^{\infty}(B(p_0,1/2))}  \leq  C_s \left\{||u_s||_{L^{\infty}(B(p_0,1))} + [f_s]_{\Tau^{0,\alpha}(B(p_0,1))} + ||g_s||_{L^{\infty}(B(p_0,1))}\right\},
\end{equation}
and
\begin{equation}\label{okk}
[\nabla_\cH u]_{\Tau^{0,\alpha}(B(p_0,1/2))}  \leq C_s \left\{||u_s||_{L^{\infty}(B(p_0,1))} + [f_s]_{\Tau^{0,\alpha}(B(p_0,1))} + ||g_s||_{L^{\infty}(B(p_0,1))}\right\},
\end{equation}
where $C_s = C(\G,\lambda,[a_{ij,s}]_{\Tau^{0,\alpha}(B(p_0,1))})>0$. Noting that \eqref{udil} gives for any $0<s\le 1$
\[
[a_{ij,s}]_{\Tau^{0, \alpha}(B(p_0,1))} = s^\alpha [a_{ij}]_{\Tau^{0, \alpha}(B(p_0,s))} \le [a_{ij}]_{\Tau^{0, \alpha}(B(p_0,s))},
\]
that by Lemma \ref{L:xjhom} and \eqref{udil} we have 
\begin{align*}
||\nabla_\cH u_s||_{L^{\infty}(B(p_0,1/2))}  & = s\ ||\nabla_\cH u||_{L^{\infty}(B(p_0,s/2))},
\\
[\nabla_\cH u_s]_{\Tau^{0, \alpha}(B(p_0,1/2))} & = s^{1+\alpha} [\nabla_\cH u]_{\Tau^{0, \alpha}(B(p_0,s/2))},
\end{align*}
and that with $f_s = (f_{1,s},...,f_{m,s})$ we have from \eqref{udil}
\[
[f_s]_{\Tau^{0, \alpha}(B(p_0,1))} = s^{1+\alpha} [f]_{\Tau^{0,\alpha}(B(p_0,s))},\ \ \ \  \  ||g_s||_{L^{\infty}(B(p_0,1))} = s^2 ||g||_{L^{\infty}(B(p_0,s))},
\]
from all these estimates and from \eqref{ok}, \eqref{okk} we reach the desired conclusion \eqref{i13}, \eqref{i133}.

\end{proof}

Next we state a H\"older continuity result near a $C^{1,\alpha}$ non-characteristic portion of the boundary that is direct consequence of the results in \cite{D}. 

\begin{proposition}\label{hol}
Suppose that $\mathbb{A}=[a_{ij}]$ be a symmetric constant-coefficient matrix. Let $\Om\subset \mathbb{G}$ be a $C^{1,\alpha}$ domain such that $p_0  \in \pa \Om$ is a non-characteristic point.   Let $u\in \mathscr{L}^{1,2}_{loc}(\Om)\cap C(\overline{\Om})$ be a weak solution of 
\begin{equation}
\begin{cases}
\sum_{i,j=1}^m X_i^\star(a_{ij} X_j u) = \sum_{i=1}^m X_i^\star  f_i + g,
\\
u= \phi\ \text{on}\ \partial \Omega,
\end{cases}
\end{equation}
 where $f^i, g \in L^{\infty}(\Om)$, $\phi \in \Tau^{0,\gamma}( \pa \Omega)$ for some $\gamma>0$. Then, there exist $r_0, C>0$ and  $\beta \in (0,1)$, depending on $\Om$, $\lambda$, $\gamma$ and $M \overset{def}{=}\underset{\Om}{\sup}\ |u|<\infty$, such that 
 \begin{equation}\label{holder}
\underset{\underset{p\neq p'}{p ,p'\in\overline{\Om \cap B_r(p_0)}}}{\sup} \frac{|u(p)-u(p')|}{d(p,p')^{\beta}} < C.
\end{equation}
\end{proposition}

\begin{proof}
By left translation, we may assume that $p_0=e$, the group identity. Since $p_0 = e$ is non-characteristic, by Definition \ref{D:char} we must have for some $i\in \{1,...,m\}$
\[
<X_i(e),\nu(e)> \not= 0.
\]
By an appropriate orthogonal transformation in the horizontal layer $\g_1$ and by the implicit function theorem, we can assume that there exist $r_0 >0$ such that $\Om \cap B(r_0)$ can be represented in the logarithmic coordinates as $\{ x_m > h(x', y)\}$, 
for a certain $h\in C^{1,\alpha}$ such that $h(0,0)=0$, $\nabla_{x'} h(0,0)=0$. Using the Lipschitz continuity of $h$ in $x', y$, similarly to the proof of Theorem 7.6 in \cite{DGP} we can assert that for sufficiently small $\kappa>0$ the gauge ball $B(\text{exp}(-se_m), \kappa s)$ is contained in $\mathbb{G} \setminus \overline{\Om}$ for all $s$ sufficiently small. This implies in a straightforward way that $\G\setminus \overline \Om$ satisfies the positive density condition at $e$ in the sense of Definition 3.6 in \cite{CGN}. Moreover, using the $C^{1, \alpha}$ character of the domain $\Om$, and by possibly taking a smaller $r_0$, we can ensure that the uniform positive density condition is satisfied at any $p \in \pa \Om \cap B(e,r_0)$. This suffices to apply Theorem 3.10 from \cite{D} to conclude the validity of the estimate \eqref{holder}.

\end{proof}

\begin{remark}\label{R:D}
We note that, although in the setting of \cite{D} both $f=(f_1, ... f_m)$ and $g$ are $\equiv 0$, nevertheless the proof there can be adapted in a straightforward way to cover the situation of Proposition \ref{hol}. 
\end{remark}

We close this section with the following smoothness result at non-characteristic points. 

\begin{theorem}\label{KNc}
 Let $\mathbb{A}=[a_{ij}]$ be a symmetric constant-coefficient matrix. Assume that $\Om$ be a $C^\infty$ domain, and let $u\in \mathscr{L}^{1,2}_{loc}(\Om)\cap C(\overline{\Om})$ be a weak solution of \eqref{re} with $f_i, g \equiv 0$.  Let  $p_0\in\pa\Om$ be  a non-characteristic point and assume that  for some neighborhood $W=B_\Ri(p_0, r_0)$ of $p_0$, we have that   $u \equiv 0$ in $\pa \Om \cap W$.  Then there exists an open neighborhood $V$ of $p_0$ depending on $W$  and $\Om$  and a positive constant $C^\star=C^\star(M, p_0)>0$, depending on $p_0$ and $M= \underset{\Om}{\sup}\ |u|$, such that
 \begin{equation}
 \label{c2u}
 \|u\|_{C^2(\overline{\Om}\cap V)}\leq C^\star.
 \end{equation}
\end{theorem}

\begin{proof}
In the following discussion we will denote by $B_\mathscr R(p,r)\subset \G$ the Riemannian ball centered at a point $p$ with radius $r$, and by $d_\mathscr R(p_1, p_2)$ the Riemannian distance between two points  $p_1, p_2\in \G$. Let $p_0\in\pa\Om$ be a non-characteristic point. It suffices to show that there exists an open neighborhood $V = B_\Ri(p_0,r_1)$ such that 
\begin{equation}\label{crux}
\nabla_\cH u\in L^{\infty}(\overline{\Om}\cap V).
\end{equation}
This obviously implies that 
\[
u\in\mathscr{L}^{1,\infty}(\overline{\Om}\cap V)\subset \mathscr{L}^{1,2}(\overline{\Om}\cap V).
\]
Once we know this, we can apply the regularity results  of \cite{De} (see also \cite{KN}) to conclude $u\in C^{\infty}(\overline{\Om}\cap V)$, and in particular  also conclude that the estimate (\ref{c2u}) holds.

Since $\mathbb{A}=[a_{ij}]$ is a constant-coefficient matrix, by a suitable linear transformation in the horizontal layer $\g_1$ we may assume without loss of generality that $a_{ij}=\delta_{ij}$ for  $i,j\in\{1,\ldots,m\}$. Thus, we have in $\Om$
\[
\Delta_\cH u = \sum^m_{i,j=1}a_{ij}X_iX_j u  = 0.
\]
Since $p_0$ is a non-characteristic point by using the $C^{2}$ character of $\Om$ we can find a sufficiently small $r_1>0$ such that $r_1\leq r_0$, a local defining function $\vf\in C^2(B_\Ri(p_0,r_1))$ for $\Om$, and a number $\gamma_0>0$ such that
\begin{equation}\label{unf0}
|\nabla_\cH \vf(p)| \ge \gamma_0
\end{equation}
for every $p\in \partial \Om \cap B_\Ri(p_0,r_1)$. Notice that, in view of \eqref{cp3}, the condition \eqref{unf0} implies that all points $p\in \pa \Om \cap B_\Ri(p_0, r_1) $ are non-characteristic. We can further assume that for every such $p$ there exists a point $\tilde p\in \G\setminus \overline \Om$ such that the geodesic ball $B_\Ri(\tilde p,r_1)$ be tangent from outside to $\pa \Om$ in $p$. This means that 
\begin{equation}\label{unf1}
\overline{B}_\Ri(\tilde p,r_1) \cap \overline{\Om}= \{p\}.
\end{equation}
Moreover  by choosing a  smaller $\gamma_0$ if needed, we can ensure that  for all $q \in \Om\cap \pa B_\Ri(p,r_1)$ one has
\begin{equation}\label{unf2}
d_\Ri(q,\tilde p) \geq (1+\gamma_0) r_1.
\end{equation}
Given any $p\in \partial \Om \cap B_\Ri(p_0,r_1)$, consider the function $\psi(q) = d_\Ri(q,\tilde p)^2$, where $\tilde p$ is as in \eqref{unf1}. Notice that by Gauss' lemma for we have that $\nabla \psi(p) = \alpha \nu(p)$, for some $\alpha = \alpha(p)\in \R$. Since $\nu(p)$ is parallel to $\nabla \vf(p)$, we conclude that $\nabla \psi(p) = \beta(p) \nabla \vf(p)$ for some $\beta(p)\in \R$ such that $|\beta(p)| \ge \ve_0>0$. By this observation and \eqref{unf0}, by possibly choosing a smaller $r_1$, we infer the existence of $\delta_0>0$ such that for all $q \in \Om \cap B_E(p,r_1)$ one has
\begin{equation}\label{t1}
|\nabla_\cH \psi(q)| \geq \delta_0.
\end{equation}

Next, we show that given any $p \in \pa \Om \cap B_\Ri(p_0,r_1)$ the following estimate holds 
\begin{equation}\label{lipu}
|u(q)|\leq K_1 d_\Ri(q,p) \quad \text{ for all}\ q \in \Om\cap  B_\Ri(p,r_1),
\end{equation}
where $K_1$ depends on $r_1, \gamma_0, \delta_0$  as in \eqref{unf0}, \eqref{t1}. By continuity, it suffices to show \eqref{lipu} for  $p= p_0$.  We denote by $\tilde p_0\in \G\setminus \overline \Om$ the point corresponding to $p_0$ as in \eqref{unf1} above, and we now let $\psi(q) =  d_\Ri(q,\tilde p_0)^2$. We note that $\psi\in C^\infty(\G)$. We introduce the function
\begin{equation}\label{funf}
f(q) \overset{def}{=} \overline{M} \left(1-\left(\frac{r_1^2}{\psi(q)}\right)^k\right),\quad\ \ \ \ \ \  q\in \Om\cap B_\Ri(p_0,r_1),
\end{equation}
where $\overline{M}$ and $k$ will be chosen later. We clearly have 
\begin{equation}\label{fgu1}
f \ge 0 = u\quad \ \ \ \ \ \ \text{on} \  \pa\Om\cap B_\Ri(p_0,r_1).
\end{equation}
If instead $q\in \Om\cap \pa B_\Ri(p_0,r_1)$, then it follows from \eqref{unf2}  that
\[
f(q) \geq \overline{M}\left(1-\left(\frac{1}{(1+\gamma_0)^2}\right)^k\right) \geq \frac{\overline{M}}{2},
\]
provided that $k\ge k_0 \overset{def}{=} \frac{\log 2}{\log(1+\gamma_0)}$. If we thus let $\overline{M}= 2 M$ and $k\ge k_0$, then we conclude that
\begin{equation}\label{fgu2}
f\ge u \quad \ \ \ \ \ \ \ \ \ \ \ \ \text{on}\  \Om\cap \pa B_\Ri(p_0,r_1).
\end{equation}
If we could show that $\Delta_\cH f \leq 0$ in $\Om\cap  B_\Ri(p_0,r_1)$, then from (\ref{fgu1}), (\ref{fgu2}) and Bony's maximum principle we would conclude that 
\begin{equation}\label{bony}
u\leq f \quad \ \ \ \ \ \ \ \ \text{in}\ \Om\cap  B_\Ri(p_0,r_1).
\end{equation}
We now observe that if $h\in C^2(\R)$, $v\in C^2(\G)$, then
\[
\Delta_\cH h(v) = h''(v)|\nabla_\cH v|^2+h'(v)\Delta_\cH v.
\]
Using this formula with $v=\psi$ and $h(t)= \overline M\left(1-(r_1^2/t)^k\right)$, we obtain
\[
\Delta_\cH f = \overline M k\frac{r_1^{2k}}{\psi^{k+2}}\left(\psi \Delta_\cH \psi - (k+1) |\nabla_\cH \psi|^2\right).
\]
Note that in $\Om\cap  B_\Ri(p_0,r_1)$ we have from \eqref{t1}  
\[
|\nabla_\cH \psi|^2 \ge \delta_0^2.
\]
Therefore, if $k$ is such that
\[
k+1 \ge \max\left\{\frac{\|\psi \Delta_\cH \psi\|_{L^{\infty}(\overline B)}}{\delta_0^2},k_0\right\},
\]
where $B\subset \G$ is a large (but fixed) geodesic ball such that $\overline \Om \subset \frac 12 B$,
then we have $\Delta_\cH f \leq 0$ in $\Om\cap  B_\Ri(p_0,r_1)$, and thus \eqref{bony} holds. In a similar way, if we take $-f$ as a lower barrier we can ensure that
\begin{equation}\label{bony2}
-f \leq u \quad\ \ \ \ \ \ \ \ \  \text{in} \ \Om\cap  B_\Ri(p_0,r_1).
\end{equation}
Combining \eqref{bony} and \eqref{bony2}, we conclude
\[
|u|\leq f \quad \ \ \ \ \ \ \ \ \  \text{in} \ \Om\cap  B_\Ri(p_0,r_1).
\]
Since $\psi\ge r_1^2$ in $\Om\cap  B_\Ri(p_0,r_1)$, by \eqref{funf} we see that the function $f$ is Lipschitz continuous in $\Om\cap  B_\Ri(p_0,r_1)$. From this fact, and $u(p_0)=f(p_0)=0$, we thus conclude that there exists a constant $K_1>0$ such that \eqref{lipu} holds for $p=p_0$, and therefore for all $p\in \pa \Om\cap  B_\Ri(p_0,r_1)$, with $q \in \Om\cap  B_\Ri(p,r_1)$.

If we now combine \eqref{lipu} with the left-hand side of \eqref{xy}, we can assert that there exists a constant $K>0$ such that 
\begin{equation}\label{lipu2}
|u(q)|\leq K_2 d_C(q,p) \quad \text{for all}\ q \in \Om\cap  B_\Ri(p,r_1),\ p\in \pa \Om\cap  B_\Ri(p_0,r_1)
\end{equation}
where $K_2>0$ depends on $r_1, \gamma_0, \delta_0$  as in \eqref{unf0}, \eqref{t1}. We stress that it is crucial that in the right-hand side of \eqref{lipu2} we have the Carnot-Carath\'eodory distance $d_C(q,p)$.
From \eqref{lipu2} and the interior Schauder estimates for solutions of $\Delta_\cH u = 0$ obtained in \cite{CDG}, arguing as in the proof of Theorem 6.6 in \cite{CGN} (see also Corollary 6.7 in the same paper) it is not difficult to show at this point that \eqref{crux} holds, i.e., $\nabla_\cH u\in L^{\infty}$ near and up to the non-characteristic portion of the boundary.

\end{proof}

\section{ Proof of Theorem \ref{main}}\label{S:mr}

The present section is devoted to proving our main result, Theorem \ref{main}. Before proceeding with its proof we  establish a  compactness lemma which is one of its crucial ingredients. Such lemma is inspired by the seminal paper of Caffarelli \cite{Ca}. 

We assume throughout the section that $\mathbb{A} =  [a_{ij}]$ be a given $m\times m$ symmetric matrix-valued function in $\G$ with real continuous coefficients and satisfying the ellipticity condition \eqref{ea0} for some $\lambda > 0$. We also denote by $\mathbb A^0$ the constant-coefficient matrix with entries $a^0_{ij} = a_{ij}(e)$, where $e\in \G$ is the group identity.  
In our ensuing discussion the gauge pseudo-ball $B(e,s)$ will be denoted by $B(s)$ for any given $s>0$. Given a bounded open set $\Om\subset \G$, with $p_0\in \pa \Om$ we use the notation  
$\V_s, \mathscr S_s$ in \eqref{VS}.
In the statement of Lemma \ref{compactness} below we assume without restriction that $p_0 = e$. We can always achieve this by a left-translation, considering the domain $\tilde \Om = L_{p_0^{-1}}(\Om)$ instead of $\Om$.

\begin{lemma}\label{compactness} 
Suppose that \eqref{ea0} hold. Assume that for a given $p_0 = e\in \pa \Om$ the set $\mathscr S_1$ be non-characteristic, and that in the logarithmic coordinates $\V_1$ is given by $\{(x,y)\mid x_m > h(x',y)\}$,
where $h \in C^{1,\alpha}$ for some $0<\alpha<1$, and $x'=(x_1,..., x_{m-1})$. Let $u\in \mathscr{L}^{1,2}_{loc}(\V_1) \cap C(\overline{\V_1})$ be a weak solution to \eqref{dp0}
in $\V_1$
with $\|u\|_{L^{\infty}(\V_1)} \leq 1$.  
Then, given $\ve >0$ there exists $\delta = \delta(\ve)>0$ such that if 
\begin{equation}\label{dilf1}
\|h\|_{C^{1,\alpha}},\ ||a_{ij}-a^{0}_{ij}||_{L^{\infty}(\V_1)},\ ||\phi||_{\Tau^{0,\alpha}(\mathscr S_1)},\ ||f_{i}||_{L^{\infty}(\V_1)},\ ||g||_{L^{\infty}(\V_1)} \leq \delta,
\end{equation}
we can find  $w\in C^2( \overline{\V_{1/2}})$  such that
\[
\|u-w\|_{L^{\infty}(\V_{1/2})}\leq \ve,
\]
with 
\[
\|w\|_{C^2(\overline{\V_{1/2}})}\leq C C^\star.
\]
Here, the constant $C>0$ is a universal constant, whereas $C^\star$ can be taken as that in the estimate \eqref{c2u} in Lemma \ref{KNc}, corresponding to $p_0 = e$ and $M=1$.
\end{lemma}

\begin{proof}

We argue by contradiction and assume the existence of $\ve_0>0$ such that for every  $k\in\mathbb{N}$ we can find:
\begin{itemize}
\item a matrix-valued function $\mathbb A^k = [a^k_{ij}]$ with continuous entries in $\G$ and satisfying \eqref{ea0},
\item a domain $\Om_k$ and, with $\V^k_1 = \Om_k\cap B(1)$ and $\mathscr S^k_1 = \pa \Om_k \cap B(1)$,
\item a solution $u_k$ to the problem
\begin{equation}\label{cacc}
\sum_{i,j=1}^m X_i^\star (a^k_{ij}X_j u_k)  = \sum_{i=1}^m X_i^\star f^k_i + g_k\ \ \ \text{in}\ \V^k_1,\ \ \ \ \ u_k  = \phi_k\ \ \ \text{on}\ \mathscr S^k_1,
\end{equation}
\end{itemize}
for which we have for any $k\in \mathbb N$:
\begin{equation}\label{uk}
\|u_k\|_{L^{\infty}(\V^k_1)}\leq 1,
\end{equation}
\begin{equation}\label{dilf11}
\|h_k\|_{C^{1,\alpha}},||a^k_{ij}-a^{0}_{ij}||_{L^{\infty}(\V^k_1)},||\phi_k||_{\Tau^{0,\alpha}(\mathscr S_1^k)},||f^k_{i}||_{L^{\infty}(\V^k_1)}, ||g_k||_{L^{\infty}(\V^k_1)} \leq \frac 1k,
\end{equation}
but for every $w\in C^2(\overline{\V^k_{1/2}})$ with  $\|w\|_{C^2(\overline{\V^k_{1/2}})} \le C C^\star$ we have
\begin{equation}\label{contr}
\|u_k-w\|_{L^{\infty}(\V^k_{1/2})}\geq \ve_0.
\end{equation}
The functions $h_k\in C^{1,\alpha}$ in \eqref{dilf11} are such that $\V^k_1$ is described in logarithmic coordinates by $\{(x,y)\mid x_m> h_k(x',y)\}$. We will show that the validity of the inequality \eqref{contr} leads to a contradiction.

Our first step is to observe that, because of the uniform bounds \eqref{uk} and \eqref{dilf11}, Proposition \ref{hol} implies the existence of constants $C, \beta>0$, depending exclusively on $\lambda, \alpha$, but not on $k$, such that
\[
\|u_k\|_{\Tau^{{0,\beta}}(\V^k_{4/5})}\leq C.
\]
Since the $u_k$'s are defined on varying domains $\V^k_1$, we need to work with functions defined on the same domain. To accomplish this, we extend each $u_k$ to a function $U_k: B(4/5)\to \R$ such that $U_k\equiv u_k$ in $\V^k_{4/5}$ and
\begin{equation}\label{ext}
\|U_k\|_{\Tau^{{0,\beta}}(B(4/5))}\leq C'\ \|u_k\|_{\Tau^{{0,\beta}}(\V^k_{4/5})}\leq C_1,
\end{equation}
for some $C_1 = C_1(\lambda,\alpha)>0$. This can be done as follows. Let $\tilde{p}=\Phi_k(p)$ be the $C^{1,\alpha}$ local diffeomorphism that straightens the portion $\mathscr S^k_1$ of $\pa \Om_k$. More precisely, $\Phi_k$ can be locally expressed in logarithmic coordinates as
\[
\Phi_k(x,y)= (x',x_m - h_k(x',y),y).
\]
We set  $v_k(\tilde{p}) = u_k\circ\Phi^{-1}_k(\tilde{p})$ and we denote by $\left(\tilde{x}',\tilde{x}_m,\tilde{y}\right)$ the logarithmic  coordinates of $\tilde{p}$. The function $v_k$ is now defined for $\tilde x_m \ge 0$. Then, we define the extension of $v_k$ to the region  $\{\tilde{x}_m<0\}$ by the classical method of extension by reflection in the following way
\begin{equation*}
V_k(\tilde{x}',\tilde{x}_m,\tilde{y})=
\begin{cases}
v_k(\tilde{x}',\tilde{x}_m,\tilde{y})\ \ \ \ \ \ \ \ \ \ \ \ \ \ \ \ \tilde x_m \ge 0,
\\
\sum^3_{i=1}c_iv_k(\tilde{x}',-\frac{\tilde{x}_m}{i},\tilde{y})\ \quad \tilde{x}_m<0,
\end{cases}
\end{equation*}
where the constants $c_1,c_2$ and $c_3$ are determined by the system of equations
\begin{equation}
\label{constants}
\sum^3_{i=1}c_i(-1/i)^m=1,\quad m=0,1,2,
\end{equation}
see e.g. p. 14 in \cite{LM}.
Having constructed $V_k$ we now define the extension $U_k$ of $u_k$ by setting $U_k=V_k\circ \Phi_k$. 

It should be clear that the sequence $\{U_k\}_{k\in \mathbb N}$ satisfies \eqref{ext}. Therefore, by the theorem of Ascoli-Arzel\`{a} we obtain a subsequence, that we still denote by $\{U_k\}_{k\in \mathbb N}$, that converges uniformly to a function $U_0\in \Tau^{{0,\beta}}(B(4/5))$. Clearly, $U_0$ satisfies
\begin{equation}\label{extw}
U_0(x',x_m,y)=
\begin{cases}
U_0(x',x_m,y)\ \ \ \ \ \ \ \ \ \ \ \ \ \ \ \ \ \ \ \ \ \ \ \ \ \ x_m\geq 0,
\\
\sum^3_{i=1} c_i U_0(x',-x_m/i,y)\ \ \ \ \ \ \ \ \ \ \ x_m< 0,
\end{cases}
\end{equation}
where the constants $c_1,c_2$ and $c_3$ are given by the system (\ref{constants}). Moreover, since by \eqref{dilf11} we have  $||\phi_k||_{\Tau^{0,1}(\V^k_1)} \leq \frac{1}{k}$ for every $k$, we  can assert that
\begin{equation}\label{zero}
U_0(x', 0, y)=0.
\end{equation}
Our next objective is to show that $U_0 \in \mathscr{L}^{1,2}_{loc}(B(4/5) \cap \{x_m>0\}) \cap C(\overline{B(4/5) \cap\{x_m>0\}})$, and moreover $U_0$ is a weak solution to the problem
\begin{equation}\label{weq}
\sum_{i,j=1}^m a^{0}_{ij}X_iX_j U_0 = 0\ \ \text{in}\ B(4/5) \cap \{x_m>0\},\ \ \ \ \ 
 U_0=0\ \  \text{on}\ B(4/5)\cap \{x_m=0\}. 
\end{equation}

If we can prove this, then since $[a^0_{ij}]$ is a constant coefficient matrix, and the portion $B(4/5)\cap \{x_m=0\}$ of the boundary of $B(4/5) \cap\{x_m>0\}$ is non-characteristic and $C^\infty$, we can apply  Theorem \ref{KNc} to conclude that  
\[
\|U_0\|_{C^2(\overline{B(1/2) \cap \{x_m>0\}})}\leq C^\star
\]
for some universal $C^\star>0$. Moreover, from  the expression of $U_0$ in \eqref{extw} we can also assert that the second derivatives in $x_m$ are continuous across $x_m = 0$, and thus in fact $U_0 \in C^2(B(1/2))$, and 
\[
\|U_0\|_{C^2(\overline{\V^k_{1/2}})} \leq \|U_0\|_{C^2(B(1/2))}\leq C C^\star,
\]
where $C>0$ is a dimensional constant. This shows that $w = U_0$ is an admissible candidate for the estimate \eqref{contr} above, and we thus have for $k\in \mathbb N$
\[
0 < \ve_0 \le \|u_k- U_0\|_{L^{\infty}(\V^k_{1/2})}.
\]
Since $u_k \to U_0$ uniformly, this is obviously a contradiction for large enough $k$'s.

In order to complete the proof of the lemma, we are thus left with establishing \eqref{weq}. To see this observe that because $\|h_k\|_{C^{1,\alpha}}\le 1/k\to 0$,  given $p \in B(4/5) \cap \{x_m>0\}$  there exist $\eta>0$ and $k_0(p)\in \mathbb N$ such that for all $k\geq k_0(p)$ we have $B(p,2\eta)\subset \V^k_1$.
In view of the uniform bounds in \eqref{uk}, \eqref{dilf11}, by the Caccioppoli inequality for the problem \eqref{cacc} we obtain for all $k\geq k_0(p)$
\begin{equation}\label{ue}
\int_{B(p,\eta)}|\nabla_\cH u_k|^2 \leq C.
\end{equation}
for some $C(\lambda,\eta)>0$ but independent of $k$. Therefore, $\{u_k\}_{k\in \mathbb N}$ has a subsequence, which we still denote by $\{u_k\}_{k\in \mathbb N}$, such that
\[
u_k\rightarrow w\ \ \text{weakly in}\ \mathscr{L}^{1,2}(B(p,\eta)),\ \  \text{and}\ \ \ 
u_k\rightarrow w\ \  \text{strongly in} \ \ L^2(B(p,\eta)).
\]

Since $\{U_k\}_{k\in \mathbb N}$ converges to $U_0$ uniformly, by uniqueness of limits we can assert that $w=U_0$ in $B(p, \eta)$. Moreover, using the uniform energy estimate for the $u_k$'s in \eqref{ue} and \eqref{dilf11} it follows  by standard weak type arguments that $U_0$ is a weak solution to  
\[
\sum_{i,j=1}^m a^0_{ij}X_i X_j U_0=0
\]
in  $B(p, \eta)$, and hence a classical solution by H\"ormander's hypoellipticity theorem in \cite{H}. By the arbitrariness of $p \in B(4/5) \cap \{x_m>0\}$ and by \eqref{zero}, we conclude that \eqref{weq} holds.

\end{proof} 
 
\medskip

With Lemma \ref{compactness}  in hands, we now proceed to proving Theorem \ref{main}.

\begin{proof}[Proof of Theorem \ref{main}] We divide it into four steps.
\begin{itemize}
\item[Step $1$:] preliminary reductions;
\item[Step $2$:] existence of the first-order Taylor polynomial at every boundary point of $\mathscr S_{1/2}$;
\item[Step $3$:] H\"older continuity of the horizontal gradient on the boundary;
\item[Step $4$:] H\"older continuity of the horizontal gradient up to the boundary.
\end{itemize}

\vskip 0.2in

\noindent \textbf{Step $1$.} Consider the function $\tilde u= u -\phi$. We note that $\tilde u$ is a weak solution of the problem
\begin{equation}\label{new}
\sum_{i,j=1}^m X_i^\star(a_{ij} X_j \tilde u) = \sum_{i=1}^m X_i^\star  \tilde f_i + g\ \text{in}\ \V_s,\ \ \ \ \tilde u=0\ \text{on}\ \mathscr S_s,
\end{equation}
where $\tilde f_i = f_i - \sum_{j=1}^m a_{ij} X_j \phi$, which again  belongs to $\Tau^{0,\alpha}(\V_s)$ in view of our assumptions on $a_{ij}, f_i$ and $\phi$. Therefore, we can assume without loss of generality that $\phi=0$ in Theorem \ref{main}.

By a left-translation we may assume that $p_0=e$ in the hypothesis of Theorem \ref{main}. Then, by scaling with respect to the family of dilations $\{\delta_{\lambda} \}_{\lambda>0}$ in \eqref{dilg}, and an appropriate  rotation of the horizontal layer $\g_1$, we may also assume without loss of generality that:
\begin{itemize}
\item $s=1$,
\item $p_0 = e$, and that
\item in the logarithmic coordinates
 $\V_1 = \Om \cap B(1)$ can be expressed as 
\begin{equation}\label{omega}
\Big\{(x',x_m,y)\mid x_m>h(x',y)\Big\},
\end{equation}
with $h(0,0)=0$,  $\nabla_{x'}h(0,0)=0$ and with $\|h\|_{C^{1,\alpha}} \leq 1$.
\end{itemize}

We denote with $\mathbb{A}(e)= \mathbb{A}^{0}= [a^{0}_{ij}]$. 
For every $0<\s \le 1$ consider now the domain $\Om_\s = \delta_{\s^{-1}}(\Om)$. In the logarithmic coordinates $\Om_\s$ is given by 
\begin{equation}\label{an1}
\Om_{\s}=\left\{(x',x_m,y_2,\ldots,y_m)\mid (\s x',\s x_m,\s^{2}y_2,\ldots,\s^{r} y_r)\in \Om\right\}.
\end{equation}
Observe that $\partial\Om_\s$ is given by 
\begin{equation}\label{xms}
x_m=h_{\s}(x',y)=h_{\s}(x',y_2,\ldots,y_m)\overset{def}{=}\frac 1\s h(\s x',\s^{2}y_2,\ldots,\s^{r} y_r).
\end{equation}
We set
\[
\U_\s = \Om_\s \cap B(\s^{-1}),\ \ \ \ \mathscr T_\s = \pa \Om_\s \cap B(\s^{-1}).
\]
Note that we have 
\begin{equation}\label{sm11}
\begin{cases}
\nabla_{x'} h_{\s}(x', y)= \nabla_{x'} h(\s x', \s^2 y_2, ....\s^r y_r),
\\
\nabla_{y_j} h_{\s}(x',y)= \s^{j-1} \nabla_{y_j} h(\s x',\s^2 y_2, ... , \s^r y_r)\ \ \ \ \  j=2,...,r,
\end{cases}
\end{equation}
and thus $\nabla h_\s(x',y)\to (\nabla_{x'} h(0,0),0)$ as $\s\to 0$. Taylor's theorem thus gives as $\s\rightarrow 0$
\begin{equation}
h_{\s}(x',y)\rightarrow \langle \nabla_{x'}h(0,0),x'\rangle=0.
\end{equation}
This implies that as $\s \to 0$
\[
\partial \Om_\s \cap B(1)\  \longrightarrow\  \{x_m=0\} \cap B(1).
\]
If we indicate with $[F]_{C^{{0,\alpha}}(U)}$  the  Euclidean H\"older seminorm of a function $F$ over a set $U$, then \eqref{sm11} implies that
\begin{align}\label{hs}
[\nabla h_\s]_{C^{{0,\alpha}}(\Om_\s\cap B(1))} & = \underset{\underset{(x',y), (\bar{x}',\bar{y})\in \Om_\s\cap B(1))}{(x',y)\not= (\bar{x}',\bar{y})}}{\sup} \frac{|\nabla h_{\s}(x',y)-\nabla h_{\s}(\bar{x}',\bar{y})|}{\left(|x-\bar{x}'|^2+|y-\bar{y}|^2\right)^{\alpha/2}}
\\
& \leq \s^{\alpha} [\nabla h]_{C^{{0,\alpha}}(\Om\cap B(1))}.
\notag
\end{align}

Similarly to the proof of Corollary \ref{intr} we now see that the rescaled function $u_{\s}(p)= u(\delta_\s p)$ solves the following problem 
\[
\sum_{i,j=1}^m X_i^\star(a_{ij,\s} X_j u_\s)= \sum_{i = 1}^m X_i^\star f_{i,\s} + g_\s\ \ \text{in}\ \U_\s,\ \ \ \ \ u_\s= \phi_\s\ \ \text{on}\ \mathscr T_\s,
\]
where
\[
a_{ij,\s}(p)= a_{ij}(\delta_\s p),\ \ \ \ f_{i,\s}(p)= \s f_i(\delta_\s p),\ \ \  \ g_\s= \s^2 g(\delta_\s p),\ \ \ \phi_\s(p) = \phi(\delta_\s p).
\]
We have  
\begin{equation}\label{cs}
\begin{cases}
[a_{ij,\s}]_{\Tau^{0,\alpha}(\U_\s)} \leq \s^{\alpha} [a_{ij}]_{\Tau^{0, \alpha}(\U_1)},
\\
||\nabla_\cH \phi_\s||_{L^{\infty}(\U_\s)} \leq \s ||\nabla_\cH \phi||_{L^{\infty}(\U_1)},\ \ \ \ [\nabla_\cH \phi_\s]_{\Tau^{0,\alpha}(\U_\s)}\leq  \s^{1+\alpha}[\nabla_\cH \phi]_{\Tau^{0,\alpha}(\U_1)},
\\
||f_{i,\s}||_{L^{\infty}(\U_\s)} \leq \s ||f_i||_{L^{\infty}(\U_1)},\ \ \ \ [f_{i,\s}]_{\Tau^{0, \alpha}(\U_\s)} \leq \s^{1+\alpha} [f_i]_{\Tau^{0, \alpha}(\U_1)},
\\
||g_\s||_{L^{\infty}(\U_\s)} \leq \s^2 ||g||_{L^{\infty}(\U_1)}.
\end{cases}
\end{equation}
From the estimates \eqref{hs} and \eqref{cs} we conclude that if for a given $\tilde \delta>0$ the scaling parameter $\s\in (0,1]$ is chosen sufficiently small, then upon substituting $\Om, a_{ij}, f_i, g, \phi$ and $u$ with $\Om_\s, a_{ij,\s}, f_{i,\s}, g_\s, \phi_\s$ and $u_\s$, we can ensure that
\begin{equation}\label{dilf}
\|h\|_{C^{1,\alpha}},\ \ [a_{ij}]_{\Tau^{0, \alpha}},\ \ ||\nabla_\cH \phi||_{L^{\infty}},\ \ [\nabla_\cH \phi]_{\Tau^{0, \alpha} },\ \ ||f_i||_{L^{\infty}},\ \ [f_i]_{\Tau^{0, \alpha}},\ \ ||g||_{L^{\infty}} \leq \tilde \delta.
\end{equation}
This reduction to small data will be important in the next step when we will want to implement Lemma \ref{compactness}.

\vskip 0.2in

\noindent \textbf{Step 2.} 
In this step we intend to show that $u$ is $\Tau^{1,\alpha}(\bar p)$ at  every $ \bar p \in \mathscr S_{1/2}$. This will be accomplished by proving that for every $\bar p\in \mathscr S_{1/2}$ there exists $P_{\bar p} \in \mathscr{P}_1$ such that for some universal constant $C_0>0$ independent of $\bar p$, one has
\begin{equation}\label{t100}
\underset{p\in \V_1}{\sup}\ \frac{|u(p) - P_{\bar p}(p)|}{d(p, \bar p)^{1+\alpha}} \leq C_0 ,\ \ \ \ \ \ ||P_{\bar p}||_{L^{\infty}(\V_1)} \leq C_0.
\end{equation}
By a left-translation it suffices to establish \eqref{t100} for $\bar p=e$. Also, by possibly normalizing the solution, we can assume without restriction that
\begin{equation}\label{indu}
||u||_{L^{\infty}(\V_1)} \le 1.
\end{equation}
Denote by $\kappa = C C^\star >0$ the constant in Lemma \ref{compactness}, and fix $\rho>0$ such that
\begin{equation}\label{rhocons}
0<\rho<(4\kappa)^{- \frac{1}{1-\alpha}}.
\end{equation}
We also let 
\begin{equation}\label{let}
\varepsilon=\frac{\rho^{1+\alpha}}{2}.
\end{equation}
Corresponding to such choice of $\ve$ we let $\delta = \delta(\ve)>0$ be the number whose existence is claimed in Lemma \ref{compactness} and let $\tilde \delta\in (0,\delta)$ be another number whose precise choice will be fixed later. As we have seen in Step 1, by appropriately choosing the scaling parameter $\s>0$, we may assume that with such $\tilde \delta$ the smallness conditions \eqref{dilf} be in force.

We first prove inductively that there exists a sequence of polynomials $\{P_\ell\}_{\ell\in \mathbb N\cup\{-1,0\}}$ in $\mathscr{P}_1$, such that for every $\ell\in \mathbb N\cup\{-1,0\}$ one has:
\begin{equation}\label{cond1}
\|u-P_\ell\|_{L^{\infty}(\Om\cap B(\rho^\ell))}\leq \rho^{\ell(1+\alpha)},\ \  \ |P_\ell| \leq C_1 \kappa,
\end{equation}
\begin{equation}
\label{cond2}
\|P_{\ell}-P_{\ell-1}\|_{L^{\infty}(B(\rho^{\ell-1}))}\leq C \rho^{(\ell-1)(1+\alpha)}, 
\end{equation}
\begin{equation}\label{cond3}
\|P_\ell\circ \delta_{\rho^\ell}\|_{\Tau^{0,1}(\pa\Om_{\rho^\ell} \cap B(1))}\leq \delta \rho^{\ell(1+\alpha)},
\end{equation}
where we recall that $\Om_\s = \delta_{\s^{-1}}(\Om)$, see \eqref{an1} above. To avoid any confusion we mention explicitly that $\delta = \delta(\ve)$ as above in the right-hand side of \eqref{cond3}, and that such number should not be confused with the non-isotropic dilations $\delta_\lambda$.
Once the validity of \eqref{cond1}, \eqref{cond2} and \eqref{cond3} is established for all $\ell$, then we can infer that  there exists $P_{e} \in \mathscr{P}_1$ such that 
\eqref{t100} holds at $\bar p=e$. This follows in a standard way  by a  real analysis argument as in the Euclidean case in \cite{Ca}  (see also \cite{CH} as well as the proof of Theorem 3.6 in \cite{GL} for adaptations to the setting of Carnot groups).   

We construct the sequence in the form of monomials $P_\ell =b_\ell x_m$ using induction. We set $b_{-1} = b_0 = 0$, which of course gives $P_{-1} = P_0 \equiv 0$.
If $\ell=0$, then by \eqref{indu} we see  that \eqref{cond1} is fulfilled, and so are \eqref{cond2}, \eqref{cond3} in a trivial way. Suppose now that, given $k \in \mathbb{N}$, the polynomials $P_1,...,P_k$ have been selected in such a way that $P_\ell =b_\ell x_m$ and conditions (\ref{cond1})-(\ref{cond3}) hold for all $\ell \leq k$. We want to prove they there exists $P_{k+1} = b_{k+1} x_m$ such that (\ref{cond1})-(\ref{cond3}) continue to hold for $\ell=k+1$. To construct $P_{k+1}$ we are going to resort to Lemma \ref{compactness}.

\medskip

Consider the following rescaled function
\begin{equation}\label{rescaled}
v(p) \overset{def}{=} \frac{\left(u-P_k\right)(\delta_{\rho^k}(p))}{\rho^{k(1+\alpha)}},\quad p\in  \tilde \Om\cap B(1). 
\end{equation}
where $\tilde \Om= \Om_{\rho^{k}}$. We immediately observe that the validity of (\ref{cond1}) for $\ell=k$ implies by rescaling that  
\begin{equation}\label{comp1}
\|v\|_{L^{\infty}(\tilde \Om\cap B(1))}\leq 1.
\end{equation}
 If we define 
\begin{equation*}
\tilde a_{ij}(p)= a_{ij}(\delta_{\rho^k} p),
\ \ \ 
\tilde f_i(p)= f_i(\delta_{\rho^k} p),\ \ \ \tilde g(p)= g (\delta_{\rho^k}p),\ \ \ \tilde P_k(p) = P_k(\delta_{\rho^k} p),
\end{equation*}
and we further let
\begin{equation}\label{comp2}
\tilde F_i = \frac{\tilde f_i - \sum_{j=1}^m \tilde a_{ij} X_j \tilde P_k}{\rho^{k\alpha}},\ \ \ \ \tilde G = \rho^{k(1-\alpha)} \tilde g,\ \ \ \ \tilde \Phi = - \rho^{-k(1+\alpha)} \tilde P_k,
\end{equation}
recalling that from Step 1 we are assuming that $\phi = 0$, we see that $v$ solves the following problem 
\begin{equation}\label{n1}
\begin{cases}
\sum_{i,j=1}^m X_i^\star (\tilde a_{ij} X_j v)= \sum_{i=1}^m X_i^\star \tilde F_i + \tilde G\ \ \ \ \text{in}\ \tilde \Om \cap B(1), 
\\
v\big|_{\pa \tilde \Om\cap B(1)}= \tilde \Phi.
\end{cases}
\end{equation}
Observe now that, since $P_k$ is of degree $1$, we have 
\[
\sum_{i, j=1}^m X_i^\star (\tilde a^{0}_{ij} X_j \tilde P_k) = 0,
\]
and therefore   
\[
X_i^\star \tilde f_i= X_i^\star (\tilde f_i - \tilde f_i(e)),\ \  \  \ \sum_{i, j=1}^m X_i^\star( \tilde a_{ij} X_j \tilde P_k)= \sum_{i, j=1}^m  X_{i}^\star \left((\tilde a_{ij}- \tilde a^{0}_{ij}) X_j \tilde P_k\right).
\]
From this observation and \eqref{n1} we conclude that $v$ also solves the problem 
\begin{equation}\label{n2}
\begin{cases}
\sum_{i,j=1}^m X_i^\star (\tilde a_{ij} X_j v)= \sum_{i=1}^m X_i^\star \tilde{\mathscr F}_i + \tilde G\ \ \ \ \text{in}\ \tilde \Om \cap B(1), 
\\
v\big|_{\pa \tilde \Om\cap B(1)}= \tilde \Phi,
\end{cases}
\end{equation}
where
\[
\tilde{\mathscr F}_i = \frac{\tilde f_i - \tilde f_i(e) - \sum_{j=1}^m (\tilde a_{ij} - \tilde a_{ij}^0) X_j \tilde P_k}{\rho^{k\alpha}}.
\]

We next want to verify that the hypothesis of Lemma \ref{compactness} are satisfied for the problem \eqref{n2} corresponding to our choice of $\ve$ in \eqref{let}. Since we have already observed that \eqref{comp1} holds for $v$, we are left with checking that \eqref{dilf1} hold for the functions
\[
\tilde h,\ \ \ \tilde a_{ij} - \tilde a_{ij}^0,\ \ \ \tilde \Phi,\ \ \ \tilde{\mathscr F}_i,\ \ \ \tilde G,
\]
where with the notation of \eqref{an1}, \eqref{xms}, we are using the letter $\tilde h = h_{\rho^{k}}$ for the function that locally describes the boundary of $\tilde \Om = \Om_{\rho^{k}} = \delta_{\rho^{-k}}(\Om)$. We note that we presently have $\tilde a_{ij}^0 = \tilde a_{ij}(e) = a_{ij}(\delta_{\rho^k} e) = a_{ij}(e) = a_{ij}^0.$ Now, since $\rho<1$, we have 
\begin{equation}\label{1}
||\tilde a_{ij} - \tilde a^{0}_{ij}||_{L^{\infty}(\tilde \Om \cap B(1))} = ||a_{ij} -  a^{0}_{ij}||_{L^{\infty}(\Om \cap B(\rho^{k}))} \le \rho^{\alpha k}\ [a_{ij}]_{\Gamma^{0,\alpha}(\Om \cap B(\rho^{k}))}  \le \tilde \delta,  
\end{equation}
where in the last inequality we have used \eqref{dilf}. Similarly, we have
\begin{equation}\label{2}
||\tilde G||_{L^{\infty}(\tilde \Om \cap B(1))} = \rho^{k(1-\alpha)} ||g||_{L^{\infty}(\Om \cap B(\rho^{k}))} \le \tilde \delta.
\end{equation}
Also  note that $\pa \tilde  \Om \cap B(1)$ can be expressed as
\begin{equation}\label{hsigma}
x_m=h_{\s}(x',y_2,\ldots,y_m)=\frac{h(\s x',\s^{2}y_2,\ldots,\s^{r}y_r)}{\s}.
\end{equation}
where $\s=\rho^{k}$. Since $\rho < 1$, by \eqref{hs} we find
\[
[\nabla \tilde h]_{C^{{0,\alpha}}(\tilde \Om\cap B(1))} = [\nabla h_\s]_{C^{{0,\alpha}}(\Om_\s\cap B(1))} \leq \s^{\alpha} [\nabla h]_{C^{{0,\alpha}}(\Om\cap B(1))} \leq \tilde \delta.
\]
Keeping in mind that $h(0,0) = 0$, this estimate easily implies
\begin{equation}\label{3}
\|\tilde h\|_{C^{1,\alpha}} \le \tilde \delta.
\end{equation}
Moreover, since 
\[
\tilde \Phi = v\big|_{\pa \tilde \Om\cap B(1)}=-\frac{P_k\circ\delta_{\rho^k}}{\rho^{k(1+\alpha)}},
\] 
we immediately obtain from \eqref{cond3} 
\begin{equation}\label{4}
\|\tilde \Phi\|_{\Tau^{0,\alpha}(\pa \tilde \Om\cap B(1))} \le \|\tilde \Phi\|_{\Tau^{0,1}(\pa \tilde \Om\cap B(1))}\leq \delta.
\end{equation}
Finally, 
\begin{align*}
||\tilde{\mathscr F}_{i}||_{L^{\infty}(\tilde \Om \cap B(1))} & = \rho^{-k\alpha} \|\tilde f_i - \tilde f_i(e) - \sum_{j=1}^m (\tilde a_{ij} - \tilde a_{ij}^0) X_j \tilde P_k\|_{L^{\infty}(\tilde \Om \cap B(1))}
\\
& \le \rho^{-k\alpha} \left\{\|\tilde f_i - \tilde f_i(e)\|_{L^{\infty}(\tilde \Om \cap B(1))} + \sum_{j=1}^m ||\tilde a_{ij} - \tilde a_{ij}^0||_{L^{\infty}(\tilde \Om \cap B(1))}\right\}
\\
& \le \rho^{-k\alpha} \left\{\|f_i - f_i(e)\|_{L^{\infty}(\Om \cap B(\rho^k))} + \sum_{j=1}^m ||a_{ij} - a_{ij}^0||_{L^{\infty}(\Om \cap B(\rho^k))}|X_j \tilde P_k|\right\}
\\
& \le [f_{i}]_{\Gamma^{0,\alpha}(\Om \cap B(\rho^{k}))} + \rho^k  [a_{im}]_{\Gamma^{0,\alpha}(\Om \cap B(\rho^{k}))}  |b_k| \le (1+ C_1 \kappa) \tilde \delta,
\end{align*}
where in the last inequality we have used \eqref{dilf} and  \eqref{cond1}.
If we now choose  
\[
\tilde \delta < \frac{\delta}{1+ C_1 \kappa},
\]
we conclude that
\begin{equation}\label{5}
||\tilde{\mathscr F}_{i}||_{L^{\infty}(\tilde \Om \cap B(1))} \le \delta.
\end{equation}

Combining \eqref{1}-\eqref{5}, we have finally proved that $v$ is a solution to the problem \eqref{n2} with $\tilde h, \tilde a_{ij} - \tilde a_{ij}^0, \tilde \Phi, \tilde{\mathscr F}_i, \tilde G$ satisfying \eqref{dilf1}. From  Lemma \ref{compactness} we can thus assert the existence of a function $w$ that solves (\ref{weq}), with $\|w\|_{C^2(B(1/2))}\leq \kappa$, and such that
\begin{equation}\label{wcomp}
\|v-w\|_{L^{\infty}( \tilde \Om\cap B(1/2))}\leq \ve.
\end{equation}
Moreover, since $w\in C^{2}$ and $w = 0$ on $B(4/5)\cap \{x_m=0\}$, by Taylor's formula and the fact that $\|w\|_{C^2(B(1/2))}\leq \kappa$ there exists $b \in \R$ with $|b|\leq \kappa$  such that
\begin{equation}\label{rho}
\|w-bx_m\|_{L^{\infty}(B(\rho))}\leq \kappa\rho^2 < \frac{\rho^{1+\alpha}}{4},
\end{equation}
where the last inequality in \eqref{rho} follows from our choice of $\rho$ in \eqref{rhocons}. From the triangle inequality, \eqref{wcomp}, \eqref{rho} and \eqref{let} we conclude that 
\begin{equation}\label{compv}
\|v-bx_m\|_{L^{\infty}(\tilde \Om\cap B(\rho))}\leq \rho^{1+\alpha}.
\end{equation}
If we now define the polynomial $Q(p) = b x_m\in \mathscr P_1$, and we keep in mind the definition \eqref{rescaled} of $v$, then we have
\begin{align}\label{QQ}
||v - Q||_{L^{\infty}(\tilde \Om\cap B(\rho))} & = \underset{p\in \tilde \Om\cap B(\rho)}{\sup} \left|\frac{\left(u-P_k\right)(\delta_{\rho^k}(p))}{\rho^{k(1+\alpha)}} - Q(p)\right|
\\
& = \rho^{-k(1+\alpha)} ||u -P_{k+1}||_{L^\infty(\Om\cap B(\rho^{k+1}))},
\notag
\end{align}
where we have let
\begin{equation}\label{pp}
P_{k+1}(p) \overset{def}{=} P_k(p)+\rho^{k(1+\alpha)} Q(\delta_{\rho^{-k}}(p)),\quad\ \ \  p\in \G.
\end{equation}
From \eqref{QQ} and \eqref{compv} we conclude that
\[
||u -P_{k+1}||_{L^\infty(\Om\cap B(\rho^{k+1}))} \le \rho^{(k+1)(1+\alpha)}.
\]
Therefore, the polynomial $P_{k+1}\in \mathscr P_1$ satisfies 
the first inequality in \eqref{cond1} for $\ell = k+1$. Furthermore, we obtain from \eqref{pp}
\[
||P_{k+1} - P_k||_{L^\infty(B(\rho^k))} = \rho^{k(1+\alpha)}  \underset{p\in B(\rho^k)}{\sup} |Q(\delta_{\rho^{-k}}(p))| = \rho^{k(1+\alpha)} ||Q||_{L^\infty(B(1))} 
\]
This proves that \eqref{cond2} is satisfied with $C = ||Q||_{L^\infty(B(1))}$. Moreover from the expression of $P_{k+1}$ in terms of $P_k$  as in \eqref{pp}, we can infer by induction that in logarithmic coordinates the polynomials $P_k$ are of the form
\begin{equation}\label{poli}
P_k(p)= b_k x_m,
\end{equation}
where
\[
|b_k| \leq \sum_{\ell = 0}^{k} \kappa  \rho^{\ell \alpha} \leq \sum_{\ell =0}^{\infty} \kappa  \rho^{\ell \alpha} \leq C_1 \kappa,
\]
where $C_1>0$ is a constant depending only on $\rho$ and $\alpha$. Therefore,  the second inequality in \eqref{cond1} also holds. It only remains to verify the condition \eqref{cond3} for $\ell = k+1$. Keeping \eqref{poli} in mind, we take $p, \bar p \in  \partial  \tilde  \Om \cap B(1)$, where $\tilde \Om = \delta_{\rho^{-(k+1)}}\Om$. Let $(x,y)$ and $(\bar{x}, \bar{y})$ respectively denote the logarithmic  coordinates of $p$ and $\bar{p}$. With $\s=\rho^{k+1}$ we have 
\[
x_m = \frac{h(\s x',\s^{2}y_2,\ldots,\s^{r}y_r)}{\s},\ \ \ \ \ \bar x_m = \frac{h(\s \bar{x}',\s^{2}\bar y_2,\ldots,\s^{r} \bar y_r)}{\s}.
\]
This gives
\begin{align}\label{sm111}
& |P_{k+1}(\delta_{\rho^{k+1}} p)- P_{k+1}(\delta_{\rho^{k+1}} \bar{p})|  = |b_{k+1}|\ |\sigma x_m - \sigma \bar x_m|
\\
& = |b_{k+1}| |h(\s x',\s^{2}y_2,\ldots,\s^{r}y_r) - h(\s \bar{x}',\s^{2}\bar y_2,\ldots,\s^{r} \bar y_r)|
\notag\\
&\leq C_1\kappa |h(\rho^{k+1}x', \rho^{2(k+1)}y_1, ...,\rho^{r(k+1)}y_r)- h(\rho^{k+1}\bar{x}', \rho^{2(k+1)}\bar{y_1}, ...,\rho^{r(k+1)}\bar{y_r})|. 
\notag
\end{align}
The triangle inequality gives
\begin{align}\label{sm1}
&|h(\rho^{k+1} x', \rho^{2(k+1)}y_2,\ldots,\rho^{r(k+1)}y_r)-h(\rho^{k+1} \bar{x}',\rho^{2(k+1)}\bar{y}_2,\ldots,\rho^{r(k+1)}r\bar{y}_r)| 
\\
& \leq  |h(\rho^{k+1} x', \rho^{2(k+1)}y_2,\ldots,\rho^{r(k+1)}y_r)-h(\rho^{k+1} \bar{x}',\rho^{2(k+1)}y_2,\ldots,\rho^{r(k+1)} y_r)| 
\notag
\\
&+|h(\rho^{k+1} \bar{x}', \rho^{2(k+1)}y_2,\ldots,\rho^{r(k+1)}y_r)-h(\rho^{k+1} \bar{x}',\rho^{2(k+1)}\bar{y}_2,\ldots,\rho^{r(k+1)}\bar{y}_r)|.
\notag
\end{align}
Since
\[
|| h||_{C^{1,\alpha}} \leq \tilde \delta,\ \ \ 
h(0, 0)=0,\ \ \ \nabla_{x'}h(0, 0)=0,
\]
for any $t>0$ one has
\begin{equation}\label{sml}
||\nabla_{x'}h||_{L^{\infty}(B(t))}\leq  \tilde \delta t^{\alpha}.
\end{equation}
By Taylor's formula and \eqref{sml} the first term in the right-hand side of \eqref{sm1} can be estimated as
\begin{align}\label{sm2}
& |h(\rho^{k+1} x', \rho^{2(k+1)}y_2,\ldots,\rho^{r(k+1)}y_r)-h(\rho^{k+1} \bar{x}',\rho^{2(k+1)}y_2,\ldots,\rho^{r(k+1)} y_r)|  \\
& \leq \tilde \delta \rho^{(k+1)(1+\alpha)}|x'- \bar{x}'| \leq  C_2 \tilde \delta \rho^{(k+1)(1+\alpha)} d(p, \bar p),
\notag
\end{align}
where in the last inequality in \eqref{sm2} we have used the left-hand side of \eqref{xy} in Lemma \ref{L:dis}, combined with the fact that $|x'-\bar x'|\le d_\mathscr R(p,\bar p)$. In a similar way the second term in the right-hand side of \eqref{sm1} can be estimated in the following way
 \begin{equation}\label{sm4}
 |h(\rho^{k+1} \bar{x}', \rho^{2(k+1)}y_2,\ldots,\rho^{r(k+1)}y_r)-h(\rho^{k+1} \bar{x}',\rho^{2(k+1)}\bar{y}_2,\ldots,\rho^{r(k+1)}\bar{y}_r)|\leq  C_4 \tilde \delta \rho^{(k+1)(1+\alpha)} d(p,\bar p).
 \end{equation}
 In \eqref{sm4} we have used the mean-value theorem and the fact that, since $\rho<1$, we have $\rho^{s(k+1)} \leq \rho^{(k+1)(1+\alpha)}$ for all $s \geq 2$. If we finally let $C_5=\max\{C_2, C_4\}$, and choose
 \[
 \tilde \delta= \min \left\{\frac{\delta}{2 C_1C_5 \kappa}, \frac{\delta}{2m^2 C_1 \kappa }\right\}, 
 \]
we conclude that \eqref{sm11}, combined with \eqref{sm1}, \eqref{sm2}  and \eqref{sm4}, ensures that \eqref{cond3} holds. This completes the proof of the estimate \eqref{t100}. 
 
\vskip 0.2in

\noindent \textbf{Step 3.}  In our subsequent discussion given any boundary point $ p\in \mathscr S_{1/2}$ we will denote by $P_{p}$ the corresponding Taylor polynomial of $u$ at $p$ whose existence has been established in Step 2. Our main objective in this step is to show that,   given $p_1, p_2 \in \mathscr S_{1/2}$, the following estimate holds for some $K_0$ universal: 
\begin{equation}\label{bdcmp}
|\nabla_\cH P_{p_1} - \nabla_\cH P_{p_2}| \leq  K_0 d(p_1,p_2)^{\alpha}.
\end{equation}
Let $t= d(p_1,p_2)$. We consider a ``non-tangential" point $p_3 \in \V_1$ at a (pseudo-)distance from $p_1$ comparable to $t$,  i.e., let $p_3$ be such that
\begin{equation}\label{nont}
d(p_3, p_1) \sim t,\ d(p_3, \pa \Om)  \sim t,
\end{equation}
where we have let $d(p,\pa \Om)= \underset{p' \in \pa \Om}{\inf} \ d(p,p')$.
Since $\mathscr S_1$ is a non-characteristic $C^{1,\alpha}$ portion of $\pa \Om$, it is possible to find such a point $p_3$. Arguing as in the proof of Theorem 7.6 in \cite{DGP}, at any scale $t$ one can find a non-tangential pseudo-ball from inside centered at $p_3$ (see the proof of Proposition \ref{hol}, where a non-tangential pseudo-ball from outside was similarly determined). The non-tangentiality of $B(p_3,at)$ means that there exists a universal $a>0$ sufficiently small (which can be seen to depend on the Lipschitz character of $\pa \Om$ near the non-characteristic portion $\mathscr S_1$), such that for some $c_0$ universal one has for all $p \in B(p_3,a t)$ 
\[
d(p,\pa \Om) \geq  c_0 t.
\]
This allows us to apply Step 2 and conclude from \eqref{t100}
that there exists a universal $K_4>0$ such that for all $p$ in $B(p_3,a t)$
\begin{equation}\label{3.4}
|u(p)  - P_{p_1}(p)| \leq K_4 t^{1+\alpha},\ \ \ \ \ \ \ \ \ |u(p) - P_{p_2}(p)| \leq K_4 t^{1+\alpha}.
\end{equation}
Now for $\ell=1, 2$  we note that  $v_\ell= u- P_{p_\ell}$ solves
\begin{equation}\label{new2}
\sum_{i,j=1}^m X_{i}^\star( a_{ij} X_j v_\ell) = \sum_{i=1}^m X_{i}^\star  F_i^\ell + g,
\end{equation}
where we have let
\[
F_i^\ell \overset{def}{=}  f_i - \sum_{j=1}^m  a_{ij} X_j P_{ p_\ell}.
\] 
Moreover  from \eqref{3.4} we see that $v_\ell$ satisfies 
\begin{equation}\label{n100}
||v_\ell||_{L^{\infty}(B(p_3,a t))} \leq K_4 t^{1+\alpha}, \ \ \ \ \ell=1, 2.
\end{equation}
With \eqref{n100} in hands, we can now use the interior estimate \eqref{i13} in Corollary \ref{intr} in the pseudo-ball $B(p_3,a t)$ obtaining the following estimate for $\ell = 1, 2$
\begin{align}\label{3.44}
& |\nabla_\cH u(p_3) - \nabla_\cH P_{p_\ell}|  = |\nabla_\cH v_\ell(p_3)| 
\\
& \leq \frac{C}{t} \left(||v_\ell||_{L^{\infty}(B(p_3,at))} + t^{1+\alpha} \underset{i=1,...,m}{\sup} [ F_i^\ell]_{\Tau^{0, \alpha}(B(p_3,at))} + t^2 ||g||_{L^{\infty}(B(p_3,at))}\right) \leq  K_5 t^{\alpha}
\notag
\end{align}
for some $K_5$ universal.  Note that in the second inequality in \eqref{3.44} we have used \eqref{n100}. From \eqref{3.44} and the triangle inequality we obtain the desired estimate \eqref{bdcmp}.

\vskip 0.2in

\noindent \textbf{Step 4.} In this final step we prove that the horizontal gradient of a weak solution to \eqref{dp0} is in $\Gamma^{0,\alpha}$ up to the boundary. As a first step we observe that we can find $\ve>0$ sufficiently small such that for any $p\in \V_\ve$ there exists $\bar p \in \mathscr S_{1/2}$ for which
\begin{equation}\label{eqt}
d(p,\bar p) = d(p,\pa \Om).
\end{equation}
To finish the proof of the theorem we will show that for all  $p,p' \in \V_\ve$ we have   \begin{equation}\label{fop}
|\nabla_\cH u(p)- \nabla_\cH u(p')|\leq  C^\star d(p,p')^{\alpha},
\end{equation}
for some  universal constant $C^\star >0$. The sought for conclusion \eqref{ap} follows from \eqref{fop} by a standard covering argument. 

Given two points  $p, p' \in \V_\ve$ we denote by $\bar p, \bar p'$  the corresponding points in $\mathscr S_{1/2}$ for which \eqref{eqt} holds.
Henceforth, we use the notation $\delta(p) = d(p,\pa \Om)$ for  $p \in \Om$. 
Without loss of generality we assume that $\delta(p)= \min \{\delta(p),\delta(p')\}$. By \eqref{t100} in Step 2, given $\bar p$ as in \eqref{eqt} we know that there exists a first-order polynomial $P_{\bar{p}}$ such that for every $q \in \V_1$ one has
\begin{equation}\label{fopb}
|u(q) - P_{\bar{p}}(q)|\leq C_2d(\bar{p},q)^{1+\alpha}.
\end{equation}
There exists two possibilities:
\begin{itemize}
\item[(i)] $d(p, p') \leq \frac{\delta(p)}{2}$;
\item[(ii)] $d(p, p') > \frac{\delta(p)}{2}$.
\end{itemize}

\medskip

\noindent \textbf{Case (i)}:  We first observe that we have $B(p, \delta(p)) \subset \Omega$. Therefore, if we let $v=u-P_{\bar p}$, then similarly to Step 3 we see that $v$ satisfies an equation of the type \eqref{new2} in $B(p, \delta(p)) \subset \Omega$. It follows from \eqref{fopb} that the following estimate holds for some $\tilde C_2>0$
\begin{equation}\label{sup}
||v||_{L^{\infty}(B(p,\delta(p))} \leq \tilde{C}_2 \delta(p)^{1+\alpha}.
\end{equation}
Since $p' \in B(p, \delta(p)/2)$, using the interior regularity estimate \eqref{i133}  in Corollary \ref{intr} and \eqref{sup}, we conclude that for some $\tilde C$ depending also on $\tilde{C}_2$, 
\begin{align}\label{sup2}
& |\nabla_\cH v(p) - \nabla_\cH v(p')|
\\
& \leq \frac{\tilde C}{\delta(p)^{1+\alpha}} \left\{\delta(p)^{1+\alpha} + \delta(p)^{1+\alpha} \underset{i=1,...,m}{\sup} [F^{\bar p}_i]_{\Tau^{0, \alpha}} + \delta(p)^2 ||g||_{L^{\infty}}\right\} d(p,p')^{\alpha}
\notag\\
& \leq  \tilde{C_1} d(p, p')^{\alpha},
\notag
\end{align}
where $F_i^{\bar p} -  f_i - \sum_{j=1}^m  a_{ij} X_j P_{\bar p}$. 
Since $P_{\bar p} \in \mathcal{P}_1$, it follows from the definition of $v$ that  
\[
|\nabla_\cH u(p) - \nabla_\cH u(p')| = |\nabla_\cH v(p) - \nabla_\cH v(p')|.
\]
From this observation and \eqref{sup2} we conclude that \eqref{fop} holds. 

\medskip

\noindent \textbf{Case (ii)}: The hypothesis in this case and \eqref{eqt} imply 
\[
d(p,\bar{p}) = d(p,\pa\Om) = \delta(p) < 2 d(p, p').
\]
Combining this observation with \eqref{t} gives  
\begin{equation}\label{n101}
d(p', \bar{p}) \leq C_0(d(p',p) + d(p,\bar{p})) \leq  C_0 (d(p',p) + 2 d(p',p)) = 3 C_0 d(p, p').
\end{equation}

Since $d(p',\bar{p}) \geq d(p',\pa \Om) = \delta(p')$, we immediately find from \eqref{n101} 
\begin{equation}\label{n103}
\delta(p') \leq 3C_0 d(p,p').
\end{equation}
From \eqref{t}, \eqref{n101} and \eqref{n103} we finally have 
\begin{equation}\label{dist}
 d(\bar{p}, \bar{p}') \leq C_0 (d(\bar p, p') + d(p', \bar{p}')) = C_0 (d(\bar p,p') + \delta(p')) \leq 6 C_0^2 d(p, p').
\end{equation}
Let now $a$ be the universal constant $a$ in the existence of a non-tangential (pseudo)-ball in Step 3. From Step 2 we infer that the following holds
\begin{equation}\label{apt}
||u - P_{\bar p}||_{L^{\infty}(B(p, a\delta(p))} \leq  \tilde K_0 \delta(p)^{1+\alpha},\ \ \ \ \ ||u - P_{\bar{p}'}||_{L^{\infty}(B(p, a\delta(p'))} \leq  \tilde K_0 \delta(p')^{1+\alpha}.
\end{equation}
Denoting $v= u - P_{\bar{p}}$, we observe that $v$  satisfies an equation of the type \eqref{new2}. Therefore, arguing as in \eqref{3.4}-\eqref{3.44} and using in $B(p,a \delta(p))$ the former estimate in \eqref{apt}  as well as the interior estimate \eqref{i13} in Corollary \ref{intr}, we obtain for some universal constant $C>0$
\begin{equation}\label{imo2}
|\nabla_\cH  u(p) - \nabla_\cH P_{\bar{p}}|= |\nabla_\cH v(p)|  \leq C \delta(p)^{\alpha} \leq C d(p, p')^{\alpha}.
\end{equation}
Note that, since we are in Case (ii), in the last inequality in \eqref{imo2} we have used $\delta(p) \leq 2 d(p, p')$. 
Arguing similarly to \eqref{imo2} we find 
\begin{equation}\label{imo1}
|\nabla_\cH  u(p') - \nabla_\cH P_{\bar{p}'}| \leq C \delta(p')^{\alpha} \leq C d(p, p')^{\alpha}.
\end{equation}
where in the last inequality in \eqref{imo1} we have this time used \eqref{n103}. 
Now, from \eqref{bdcmp} and  \eqref{dist} we have
\begin{equation}\label{finali}
|\nabla_\cH P_{\bar p} - \nabla_\cH P_{\bar{p}'}| \leq K_0 d(\bar p, \bar{p}')^{\alpha} \le C d(p, p')^{\alpha}.
\end{equation}
Applying the triangle inequality with the estimates \eqref{imo2}, \eqref{imo1} and \eqref{finali} we finally conclude that \eqref{fop} holds.  This completes the proof of the theorem. 

\end{proof}

\end{document}